\documentclass[a4paper]{amsart}
\usepackage[all]{xy}

\usepackage{pstricks-add}
\usepackage{float}

\usepackage[colorlinks=true]{hyperref}
\usepackage{enumerate}
\usepackage{amssymb}
\usepackage{amsbsy}
\usepackage{comment}
\usepackage{aurical}

\newtheorem{thm}{Theorem}[section]

\newtheorem{prop}[thm]{Proposition}
\newtheorem{lem}[thm]{Lemma}
\newtheorem{cor}[thm]{Corollary}

\theoremstyle{definition}
\newtheorem{defn}[thm]{Definition}

\theoremstyle{remark}
\newtheorem{rem}[thm]{Remark}

\newcommand{\Rr}{\mathbb R}

\renewcommand{\d}{\mathrm d}                           

\renewcommand{\d}{\mathrm d}               

\def\bb#1#2{\left\{#1,#2\right\}}

\begin{document}
\title{Hypersymplectic structures on Courant algebroids}

\author{P. Antunes}
\address{CMUC, Department of Mathematics, University of Coimbra, 3001-501 Coimbra, Portugal}
\email{pantunes@mat.uc.pt}
\author{J.M. Nunes da Costa}
\address{CMUC, Department of Mathematics, University of Coimbra, 3001-501 Coimbra, Portugal}
\email{jmcosta@mat.uc.pt}

\begin{abstract}
We introduce the notion of hypersymplectic structure on a Courant algebroid and we prove the existence of a one-to-one correspondence between hypersymplectic and hyperk\"{a}hler structures. This correspondence provides a simpler way to define a hyperk\"{a}hler structure on a Courant algebroid. We show that hypersymplectic structures on Courant algebroids encompass hyperk\"{a}hler and  hyperk\"{a}hler structures with torsion on Lie algebroids. In the latter, the torsion existing at the Lie algebroid level is incorporated in the Courant structure. Cases of hypersymplectic structures on Courant algebroids which are doubles of Lie, quasi-Lie and proto-Lie bialgebroids are investigated.
\end{abstract}

\maketitle


\


%
\section{Introduction}             %
\label{section_introduction}           %

In the past years, hyperstructures on Courant algebroids deserved the attention of several authors.
Namely, we mention Bursztyn {\em et al.} \cite{BCG08} who discussed hyperk\"{a}hler structures and Sti\'enon \cite{Stienon} for the case of hypercomplex structures. In the present article we introduce and study hypersymplectic structures and, more generally, $\boldsymbol{\varepsilon}$-hypersymplectic structures on Courant algebroids, a notion that encompasses hypercomplex structures, because the para-complex case is also covered, and hyperk\"{a}hler structures. In fact,
 a very interesting feature of hypersymplectic structures on Courant algebroids is that they are in a one-to-one correspondence with hyperk\"{a}hler structures. An important point to notice, which is a direct consequence of the existence of this correspondence, is that one gets a simpler way to define a hyperk\"{a}hler structure because, contrary to the case of hyperk\"{a}hler structures, our definition of hypersymplectic structure does not require the existence \emph{ab initio} of a pseudo-metric on the Courant algebroid. As we shall see in Section~\ref{section_5}, the pseudo-metric is constructed out of the given endomorphisms.

The basic example of Courant algebroid is the vector bundle $A\oplus A^*$ equipped with the Dorfmann bracket (or its skew-symmetrization, called the Courant bracket). This Courant structure on  $A \oplus A^*$ is the double of a Lie bialgebroid structure on $(A,A^*)$, where $A^*$ is the null Lie algebroid. Doubles of more general cases of Lie bialgebroids structures on $(A,A^*)$ are also Courant structures on $A\oplus A^*$. More generally, doubles of quasi-Lie bialgebroids and of proto-Lie bialgebroids structures on $(A,A^*)$ determine Courant structures on $A\oplus A^*$ \cite{YKS05, roy}.
In the Lie algebroid setting, and inspired by hypersymplectic structures on manifolds, defined by Xu in \cite{Xu97}, we have introduced  in \cite{AC13} the notion of hypersymplectic structure (see also \cite{A10}). A hypersymplectic structure on a Lie algebroid $A$ is a triplet $(\omega_1, \omega_2, \omega_3)$ of symplectic forms  on $A$, such that the square of the transition morphisms, endomorphisms of $A$ constructed out of the 2-forms $\omega_i$ and their inverses,  is equal to $\pm {\rm id}_A$. The extension of the theory of hypersymplectic structures to Courant algebroids is not straightforward, since the notion of symplectic section on a Courant algebroid is not known. However, inspired on the generalized complex geometry \emph{\`a la Hitchin} \cite{hitchin}, where a $2$-form on a manifold $M$ gives rise to an endomorphism of $TM \oplus T^*M$, we replace the morphism $\omega^\flat$ associated to a $2$-form $\omega$ on $A$ by an endomorphism of $A \oplus A^*$. More precisely, the idea which is behind our results, is to associate to each triplet $(\omega_1, \omega_2, \omega_3)$ of non-degenerate $2$-forms  on a Lie algebroid $A$, with inverse $(\pi_1, \pi_2,\pi_3) \in (\Gamma(\wedge^2 A))^3$, a triplet of endomorphisms
$\mathcal{S}_{i}=\left[
                    \begin{array}{ccc}
                    0&\ &\varepsilon_i\,\pi_i^{\sharp}\\
                    \omega_i^{\flat}&\ &0
                    \end{array}
\right]$, \, $\varepsilon_i=\pm 1, i=1,2,3$, of the vector bundle $A \oplus A^*$, and introduce an appropriated notion of hypersymplectic structure on a Courant algebroid in such a way that the following holds: $(\omega_1, \omega_2, \omega_3)$ is a hypersymplectic structure on a Lie algebroid $(A, \mu)$ if and only if $(\mathcal{S}_{1}, \mathcal{S}_{2}, \mathcal{S}_{3})$ is a hypersymplectic structure on the Courant algebroid $(A \oplus A^*, \mu)$. However, while considering the vector bundle $A \oplus A^*$ equipped with a Courant structure which is of type $\mu + \psi$, with $\psi$ a trivector on $A$, we realize that the previous equivalence fails. In fact, the structure that has to be considered on $A$ is a hypersymplectic structure with torsion, a notion that we study in a separate article \cite{AC14_a} and which is related to hyperk\"{a}hler structures with torsion. Hyperk\"{a}hler structures with torsion on manifolds, also known as HKT structures,  first appeared in \cite{HP96} in relation with sigma models in string theory and, since then, HKT and other geometries with torsion caught the interest of many physicists and mathematicians.

In the current article most results are established on the more general pre-Courant framework, and hold   without any change, in the Courant algebroid case. The exception is when we deal with hypersymplectic structures with torsion on Lie algebroids (Section \ref{section_8}), where an extra condition must be considered when we pass from pre-Courant to Courant algebroids. We show that, although hypersymplectic and hypersymplectic structures with torsion on Lie algebroids
are different in nature, when we look at them in the pre-Courant algebroid setting, they become of the same type. Roughly speaking, when one goes from Lie to pre-Courant algebroids, the torsion carried by the hypersymplectic structure on the Lie algebroid passes to the pre-Courant structure itself. More precisely, we prove that having a hypersymplectic structure $(\omega_1, \omega_2, \omega_3)$ on $A$, with or without torsion, is equivalent to having a hypersymplectic structure  $(\mathcal{S}_{1}, \mathcal{S}_{2}, \mathcal{S}_{3})$ on $A \oplus A^*$, with a suitable pre-Courant structure. More involved situations are those where, besides the structures considered on $A$, the vector bundle $A^*$ itself is endowed with a hypersymplectic structure, with or without torsion, determined by  $(\pi_1, \pi_2,\pi_3)$. We also prove that, under some conditions, this is equivalent to $(\mathcal{S}_{1}, \mathcal{S}_{2}, \mathcal{S}_{3})$ being a hypersymplectic structure on $A \oplus A^*$.

Besides the Introduction, the article contains eight sections. Since many of the computations are done using the big bracket,  Section \ref{section_1} contains a brief review of the supergeometric setting as well as the main notions around the Courant and pre-Courant algebroid definitions. Section \ref{section_3} is devoted to some properties of the Nijenhuis torsion on pre-Courant algebroids, that are used in the remaining sections. To our knowledge some of these properties (Propositions \ref{prop_T_theta(N)=0<->T_theta_N(N)=0} and \ref{prop_T_theta(J)=0<->T_theta_I(J)=0}) are new. In Section \ref{section_2} we introduce the notion of $\boldsymbol{\varepsilon}$-hypersymplectic structure on a pre-Courant algebroid (Definition~\ref{def_epsilonHS_pre-Courant}) and we explore the properties of the morphisms induced by this structure.
Sections \ref{section_4} and \ref{section_5} treat the case $\varepsilon_1 \varepsilon_2 \varepsilon_3=-1$. The main result of Section \ref{section_4} is that the transition morphisms $\mathcal{T}_{i}$ are Nijenhuis (Theorem~\ref{T_i_Nijenhuis}).
We also show that  $(\mathcal{S}_{1}, \mathcal{S}_{2}, \mathcal{S}_{3})$ is a hypersymplectic structure on a pre-Courant algebroid $(E, \Theta)$ if and only if it is hypersymplectic for the pre-Courant structure on $E$ deformed by $\mathcal{T}_{i}$ or by $\mathcal{S}_{i}$ (Theorem~\ref{thm_deformation_Theta}). In Section \ref{section_5} we prove a one-to-one correspondence theorem between hypersymplectic and hyperk\"{a}hler structures on a pre-Courant algebroid (Theorem~\ref{thm_1_1_corresp}). Moreover, we show how the transition morphisms $\mathcal{T}_{i}$ can take the role of the morphisms $\mathcal{S}_{i}$ to define a new hypersymplectic structure on the pre-Courant algebroid (Theorem~\ref{thm_invariant_by_rotation}).
Sections \ref{section_6}, \ref{section_7} and \ref{section_8} are devoted to examples of hypersymplectic structures on $A \oplus A^*$, equipped with several (pre-)Courant structures.
We start with the simplest case in Section \ref{section_6}. We prove that $(\omega_1, \omega_2, \omega_3)$ is a hypersymplectic structure on a Lie algebroid $(A, \mu)$ if and only if $(\mathcal{S}_{1}, \mathcal{S}_{2}, \mathcal{S}_{3})$ is a hypersymplectic structure on the Courant algebroid $(A \oplus A^*, \mu)$ (Theorem~\ref{structure_A+A*_mu}). In Section \ref{section_7} the Courant structure on $A \oplus A^*$ is the double of a Lie bialgebroid $((A,A^*), \mu, \gamma)$ and we prove that $(\mathcal{S}_{1}, \mathcal{S}_{2}, \mathcal{S}_{3})$ is a hypersymplectic structure on $(A \oplus A^*, \mu + \gamma)$ if and only if  $(\omega_1, \omega_2, \omega_3)$ is a hypersymplectic structure on $(A, \mu)$ and $(\pi_1, \pi_2, \pi_3)$ is a hypersymplectic structure on $(A^*, \gamma)$ (Theorem~\ref{structure_A+A*_mu+gamma}). The particular case of a triangular Lie bialgebroid is also considered (Corollary~\ref{cor_structure_A+A*_mu+mu_pi}). The class of examples we give in Section \ref{section_8}, deal with the notion of hypersymplectic structure with torsion on a Lie algebroid. This is a structure that generalizes the hypersymplectic case, where the non-degenerate $2$-forms $\omega_i$ are not closed but satisfy the condition $N_1 \d \omega_1= N_2 \d \omega_2=N_3 \d \omega_3$, with $N_i$ the transition morphisms. We show that having a hypersymplectic structure with torsion $(\omega_1, \omega_2, \omega_3)$ on a Lie algebroid $(A,\mu)$ is equivalent to $(\mathcal{S}_{1}, \mathcal{S}_{2}, \mathcal{S}_{3})$ being a hypersymplectic structure on the pre-Courant algebroid $(A \oplus A^*, \mu+\psi)$, with $\psi \in \Gamma(\wedge^3A)$ (Proposition~\ref{HS_preCourant_structure_mu+gamma+psi}). The corresponding result on the Courant algebroid $(A \oplus A^*, \mu+\psi)$, which is the double of  the quasi-Lie bialgebroid  $((A, A^*), \mu, 0, \psi)$, requires the bivectors $\pi_i$ to be weak-Poisson with respect to $\mu$ (Theorem~\ref{HS_Courant_structure_mu+gamma+psi}).  The last case that we treat, which is the more general one,  is when both Lie algebroids $(A, \mu)$ and $(A^*, \gamma)$ are equipped with hypersymplectic structures with torsion.  We show that this is equivalent to having a hypersymplectic structure on the pre-Courant algebroid $(A \oplus A^*, \mu+\gamma+\psi+\phi)$, with $\psi \in \Gamma(\wedge^3A)$ and $\phi \in \Gamma(\wedge^3A^*)$ (Proposition \ref{HS_preCourant_structure_mu+gamma+psi+phi}). As before,  the corresponding result on the Courant algebroid $(A \oplus A^*, \mu+\gamma +\psi +\phi)$, which is the double of  the  proto-Lie bialgebroid $((A, A^*), \mu, \gamma, \psi, \phi)$, requires some extra conditions on the $\omega_i$'s and on the $\pi_i$'s (Theorem~\ref{last_thm}).



%
\section{Preliminaries on Courant algebroids}         
\label{section_1}                                     

%
We begin this section by introducing the supergeometric setting, following the same approach as in \cite{voronov,royContemp} (see also \cite{A10}). Given a vector bundle $A \to M$, we denote by $A[n]$ the graded manifold obtained by shifting the fibre degree by $n$. The graded manifold $T^*[2]A[1]$ is equipped with a canonical symplectic structure which induces a Poisson bracket on its algebra of functions $\mathcal{F}:=C^\infty(T^*[2]A[1])$. This Poisson bracket is sometimes called the \emph{big bracket} (see \cite{YKS05}).

Let us describe locally this Poisson algebra. Fix local coordinates $x_i, p^i,\xi_a, \theta^a$, $i \in \{1,\dots,n\}, a \in \{1,\dots,d\}$, in $T^*[2]A[1]$, where $x_i,\xi_a$ are local coordinates on $A[1]$ and $p^i, \theta^a$ are their associated moment coordinates. In these local coordinates, the Poisson bracket is given by
 $$ \{p^i,x_i\}=\{\theta^a,\xi_a\}=1,  \quad  i =1, \dots, n, \, \, a=1, \dots , d, $$
while all the remaining brackets vanish.

The Poisson algebra of functions $\mathcal{F}$ is endowed with an $(\mathbb{N} \times \mathbb{N})$-valued bidegree. We
define this bidegree (locally but it is well defined globally, see \cite{voronov, royContemp}) as follows: the coordinates on the base
manifold $M$, $x_i$, $i \in \{1,\dots,n\}$, have bidegree $(0,0)$, while the coordinates on the fibres, $\xi_a$, $a \in \{1,\dots,d\}$,
have bidegree $(0,1)$ and their associated moment coordinates, $p^i$ and $\theta^a$, have bidegree $(1,1)$ and $(1,0)$, respectively.
We denote by $\mathcal{F}^{k,l}$ the space of functions of bidegree $(k,l)$ and we verify that the big bracket has bidegree $(-1,-1)$, i.e.,
$$\{\mathcal{F}^{k_1,l_1},\mathcal{F}^{k_2,l_2}\}\subset \mathcal{F}^{k_1+k_2-1,l_1+l_2-1}.$$

This construction is a particular case of a more general one \cite{royContemp} in which we consider a vector bundle $E$
equipped with a fibrewise non-degenerate symmetric bilinear form $\langle.,.\rangle$.
In this more general setting, we consider the graded symplectic manifold $\mathcal{E}:=p^*(T^*[2]E[1])$,
which is the pull-back of $T^*[2]E[1]$ by the map $p:E[1] \to E[1]\oplus E^*[1]$ defined by $X \mapsto (X, \frac{1}{2}\langle X,.\rangle)$.
 We denote by $\mathcal{F}_{E}$ the graded algebra of functions on $\mathcal{E}$, i.e.,
 $\mathcal{F}_{E}:=C^\infty(\mathcal{E})$. The algebra $\mathcal{F}_{E}$ is equipped with
the canonical Poisson bracket, denoted by $\{.,.\}$, which has degree $-2$.
Notice that $\mathcal{F}_{E}^0=C^\infty(M)$ and $\mathcal{F}_{E}^1=\Gamma(E)$. Under these
identifications, the Poisson bracket of functions of degrees $0$ and $1$ is given by
$$\{f,g\}=0,\; \; \{f, X\}=0 \quad {\hbox{and}} \quad \{X,Y\}=\langle X,Y \rangle,$$
for all $X,Y \in \Gamma(E)$ and $f,g \in C^\infty(M)$.

 When $E:=A\oplus A^*$ (with $A$ a vector bundle over $M$) and when $\langle .,.\rangle$ is the usual symmetric bilinear form:
 \begin{equation}\label{eq:canbin} \langle X + \alpha, Y + \beta \rangle = \alpha(Y)+ \beta(X), \hbox{  } \,\,\, \forall X,Y \in \Gamma(A), \alpha, \beta \in \Gamma(A^*), \end{equation}
the algebras $\mathcal{F}=C^\infty(T^*[2]A[1])$ and $\mathcal{F}_{A\oplus A^*}$ are isomorphic Poisson algebras \cite{royContemp} and the two constructions above coincide.

\begin{defn}\cite{ALC11} \label{def_courant}
 A \emph{pre-Courant} structure on $(E, \langle\cdot,\cdot\rangle)$ is a pair $(\rho, [\cdot,\cdot])$, where  $\rho:E\to TM$ is a morphism of vector bundles called the \emph{anchor},
 and $[\cdot,\cdot]:\Gamma(E)\times \Gamma(E)\to \Gamma(E)$ is a $\mathbb{R}$-bilinear (non necessarily skew-symmetric) bracket, called
the \emph{Dorfman bracket}, satisfying the relations
\begin{equation} \label{pre_Courant1}
\rho(X)\cdot\langle Y,Z\rangle=\langle[X,Y],Z\rangle +  \langle Y,[X,Z]\rangle
\end{equation}
and
\begin{equation} \label{pre_Courant2}
\rho(X)\cdot \langle Y,Z\rangle=\langle X, [Y,Z]+ [Z,Y]\rangle,
\end{equation}
for all $X,Y,Z \in \Gamma(E)$.
\end{defn}

\noindent From (\ref{pre_Courant1}) and (\ref{pre_Courant2}), we obtain the Leibniz rule~\cite{YKS05}
$$[X, fY]=f[X,Y]+ (\rho(X).f)Y,$$ for all $X,Y \in \Gamma(E)$ and $f \in C^\infty(M)$.
If a pre-Courant structure $(\rho, [\cdot,\cdot])$ satisfies the Jacobi identity,
$$[X,[Y,Z]] =[[X,Y],Z] + [Y,[X,Z]],$$
for all $X,Y,Z \in \Gamma(E)$, then the pair $(\rho, [\cdot,\cdot])$ is called a \emph{Courant} structure on \mbox{$(E,\langle\cdot,\cdot\rangle)$}.

There is a one-to-one correspondence between pre-Courant structures on $(E, \langle\cdot,\cdot\rangle)$ and elements in $\mathcal{F}_E^3$. The anchor and Dorfman bracket associated to a given $\Theta\in \mathcal{F}_E^3$ are defined, for all $X,Y \in \Gamma(E)$ and $f \in C^\infty(M)$, by the derived bracket expressions
\begin{equation}\label{eq_derived_bracket_expressions}
  \rho(X)\cdot f=\{\{X,\Theta\},f\} \quad {\hbox{and}} \quad {[X,Y]=\{\{X,\Theta\},Y\}}.
\end{equation}

The next theorem shows how a Courant structure can be defined in the supergeometric setting.
\begin{thm} \cite{royContemp}
There is a $1-1$ correspondence between Courant structures on \mbox{$(E,\langle.,.\rangle)$} and functions $\Theta \in \mathcal{F}_E^3$ such that \mbox{$\{\Theta,\Theta\}=0$}.
\end{thm}
\noindent If $\Theta$ is a (pre-)Courant structure on $(E, \langle\cdot,\cdot\rangle)$, then the triple $(E, \langle\cdot,\cdot \rangle, \Theta)$ is called a \emph{(pre-)Courant algebroid}.
For the sake of simplicity, we often denote a (pre-)Courant algebroid by the pair $(E, \Theta)$ instead of the triple $(E, \langle\cdot,\cdot \rangle, \Theta)$.

When $E= A \oplus A^*$ and $\langle\cdot,\cdot\rangle$ is the usual symmetric bilinear form (\ref{eq:canbin}), a pre-Courant structure $\Theta \in \mathcal{F}_E^3$ can be decomposed using the bidegrees:
$$\Theta=\mu + \gamma + \phi + \psi,$$
with $\mu \in \mathcal{F}_{A \oplus A^*}^{1,2}, \gamma \in \mathcal{F}_{A \oplus A^*}^{2,1}, \phi \in \mathcal{F}_{A \oplus A^*}^{0,3}=\Gamma(\wedge^3 A^*)$ and $\psi \in \mathcal{F}_{A \oplus A^*}^{3,0}=\Gamma(\wedge^3 A)$.
We recall from~\cite{roy} that, when $\gamma = \phi = \psi =0$, $\Theta$ is a Courant structure on $(A \oplus A^*, \langle\cdot,\cdot\rangle)$ \emph{if and only if} $(A,\mu)$ is a Lie algebroid; the anchor and the bracket of the Lie algebroid $(A, \mu)$ are given by (\ref{eq_derived_bracket_expressions}), where a section $X$ of $A$ is identified with $X\oplus 0 \in \Gamma(A\oplus A^*)$.
When $\phi = \psi =0$, $\Theta$ is a Courant structure on $(A \oplus A^*, \langle\cdot,\cdot\rangle)$ \emph{if and only if} $((A,A^*), \mu, \gamma)$ is a Lie bialgebroid and when $\phi=0$ (resp. $\psi=0$), $\Theta$ is a Courant structure on $(A \oplus A^*, \langle\cdot,\cdot\rangle)$ \emph{if and only if} $((A,A^*), \mu, \gamma, \psi)$ (resp. $((A^*,A), \gamma, \mu, \phi)$) is a quasi-Lie bialgebroid. In the more general case, $\Theta=\mu + \gamma + \phi + \psi$ is a Courant structure \emph{if and only if} $((A,A^*), \mu, \gamma, \psi, \phi)$  is a proto-Lie bialgebroid.


\section{Nijenhuis torsion on pre-Courant algebroids}
\label{section_3}
%

Let $(E,\langle \cdot,\cdot \rangle,\Theta)$ be a pre-Courant algebroid with anchor and Dorfman bracket defined by (\ref{eq_derived_bracket_expressions}). Given an endomorphism $\mathcal{I}:E \to E$, we define a \emph{deformed} pre-Courant algebroid structure $(\rho_{\mathcal{I}}, [\cdot,\cdot]_{\mathcal{I}})$ on $E$ by setting
$$\left\{
  \begin{array}{l}
    \rho_{\mathcal{I}}=\rho \circ {\mathcal{I}}\\
    {[X,Y]_{\mathcal{I}}} =[{\mathcal{I}}X,Y]+[X,{\mathcal{I}}Y]-{\mathcal{I}}[X,Y], \quad \forall X,Y \in \Gamma(E).
  \end{array}
\right.$$
The deformation of $(\rho_{\mathcal{I}}, [\cdot,\cdot]_{\mathcal{I}})$ by an endomorphism ${\mathcal{J}}$ of $E$ is denoted by $(\rho_{{\mathcal{I}},{\mathcal{J}}}, [\cdot,\cdot]_{{\mathcal{I}},{\mathcal{J}}})$.
The \emph{concomitant} $C_\Theta({\mathcal{I}},{\mathcal{J}})$ of two endomorphisms ${\mathcal{I}}$ and ${\mathcal{J}}$, on a pre-Courant algebroid $(E,\langle \cdot,\cdot \rangle,\Theta)$, is a $\mathbb{R}$-bilinear map $\Gamma(E)\times\Gamma(E)\to\Gamma(E)$ defined, for all sections $X,Y$ of $E$, by
$$C_\Theta({\mathcal{I}},{\mathcal{J}})(X,Y):=[X,Y]_{{\mathcal{I}},{\mathcal{J}}}+[X,Y]_{{\mathcal{J}},{\mathcal{I}}}.$$
Recall that an endomorphism ${\mathcal{I}}:E \to E$ on a pre-Courant algebroid $(E,\langle \cdot,\cdot \rangle,\Theta)$ is a {\em Nijenhuis morphism} if its Nijenhuis torsion ${\text{\Fontlukas T}}_{\Theta}{\mathcal{I}}$ vanishes, where
\begin{equation}\label{def_Nijenhuis_torsion}
    {\text{\Fontlukas T}}_{\Theta}{\mathcal{I}}(X,Y)=[\mathcal{I}X, \mathcal{I}Y]-\mathcal{I}\left([X,Y]_{\mathcal{I}}\right)
        =\frac{1}{2}\Big([X,Y]_{\mathcal{I},\mathcal{I}}-[X,Y]_{\mathcal{I}^2}\Big),
\end{equation}
for all $X,Y \in \Gamma(E)$.

Given an endomorphism ${\mathcal{I}}:E \to E$, the transpose morphism ${\mathcal{I}}^*:E^*\simeq E \to E^*\simeq E$ is defined by $\langle {\mathcal{I}}^*u,v \rangle = \langle u,{\mathcal{I}}v \rangle$ for all $u,v \in E$. The morphism $\mathcal{I}$ is \emph{orthogonal} if $\mathcal{I}\circ \mathcal{I}^*={\rm id}_E$. If ${\mathcal{I}}={\mathcal{I}}^*$ (resp. ${\mathcal{I}}=-{\mathcal{I}}^*$), the morphism ${\mathcal{I}}$ is said to be \emph{symmetric} (resp. \emph{skew-symmetric}).

When ${\mathcal{I}}$ is skew-symmetric, the deformed pre-Courant structure $(\rho_{\mathcal{I}}, [\cdot,\cdot]_{\mathcal{I}})$ corresponds to the function $\Theta_{{\mathcal{I}}}:=\{{\mathcal{I}},\Theta\}\in\mathcal{F}_E^{3}$, (via (\ref{eq_derived_bracket_expressions})).
The deformation of $\Theta_{\mathcal{I}}$ by a skew-symmetric morphism ${\mathcal{J}}$ is denoted by
$\Theta_{{\mathcal{I}},{\mathcal{J}}}$, i.e. $\Theta_{{\mathcal{I}},{\mathcal{J}}}=\{{\mathcal{J}},\{{\mathcal{I}}, \Theta \}\}.$
When ${\mathcal{I}}$ and ${\mathcal{J}}$ are skew-symmetric endomorphisms of $E$, the concomitant $C_\Theta({\mathcal{I}},{\mathcal{J}})$ is an element of $\mathcal{F}_E^3$ and may be defined as~\cite{ALC11}:
\begin{equation}  \label{def_conc}
C_\Theta({\mathcal{I}},{\mathcal{J}})=\Theta _{{\mathcal{I}},{\mathcal{J}}}+\Theta_{{\mathcal{J}},{\mathcal{I}}}.
\end{equation}

When $\mathcal{I}$ is skew-symmetric and satisfies ${\mathcal{I}}^2= \lambda\, {\rm id}_E$, for some $\lambda \in \Rr$,
we have \cite{grab,A10}
\begin{equation} \label{supergeometric_torsion}
{\text{\Fontlukas T}}_{\Theta}{\mathcal{I}}= \frac{1}{2}(\Theta_{{\mathcal{I}},{\mathcal{I}}}-\lambda \Theta).
\end{equation}
 If ${\mathcal{I}}^2= - {\rm id}_E$ (resp. ${\mathcal{I}}^2= {\rm id}_E$) then ${\mathcal{I}}$ is said to be an \emph{almost complex}  (resp. \emph{almost para-complex}) \emph{structure}. If moreover ${\text{\Fontlukas T}}_{\Theta}{\mathcal{I}}=0$, then ${\mathcal{I}}$ is a \emph{complex} (resp. \emph{para-complex}) \emph{structure}.




Let us recall a result from \cite{ALC11}.

\begin{prop}  \label{torsion_composition}
    Let ${\mathcal{I}}$ and ${\mathcal{J}}$ be two anti-commuting endomorphisms on a pre-Courant algebroid $(E, \Theta)$. Then, for all sections $X$ and $Y$ of $E$,
    \begin{multline}
        2\ {\text{\Fontlukas T}}_{\Theta}({\mathcal{I}}\circ {\mathcal{J}})(X,Y) = \bigg({\text{\Fontlukas T}}_{\Theta}{\mathcal{I}}({\mathcal{J}}X, {\mathcal{J}}Y) - {\mathcal{J}}\left({\text{\Fontlukas T}}_{\Theta}{\mathcal{I}}({\mathcal{J}}X,Y) + {\text{\Fontlukas T}}_{\Theta}{\mathcal{I}}(X,{\mathcal{J}}Y)\right) -\\- {\mathcal{J}}^2({\text{\Fontlukas T}}_{\Theta}{\mathcal{I}}(X,Y)) \bigg) + \underset{{\mathcal{I}},{\mathcal{J}}}{\circlearrowleft},
    \end{multline}
 where $ \underset{{\mathcal{I}},{\mathcal{J}}}{\circlearrowleft}$ stands for permutation of ${\mathcal{I}}$ and ${\mathcal{J}}$.
In particular, if ${\mathcal{I}}$ and ${\mathcal{J}}$ have vanishing Nijenhuis torsion then so has ${\mathcal{I}}\circ {\mathcal{J}}$.
\end{prop}

\begin{prop}\label{prop_T_theta(N)=0<->T_theta_N(N)=0}
    Let ${\mathcal{I}}$ be an endomorphism on a pre-Courant algebroid $(E, \Theta)$, such that ${\mathcal{I}}^2=\lambda\,{\rm id}_E$, for some $\lambda \in \mathbb{R}\!\setminus\!\{0\}$. Then, \mbox{${\text{\Fontlukas T}}_{\Theta}{\mathcal{I}}=0 \Leftrightarrow {\text{\Fontlukas T}}_{\Theta_{\mathcal{I}}}{\mathcal{I}}=0$}.
\end{prop}
\begin{proof}
    The statement \mbox{${\text{\Fontlukas T}}_{\Theta}{\mathcal{I}}=0 \Rightarrow {\text{\Fontlukas T}}_{\Theta_{\mathcal{I}}}{\mathcal{I}}=0$} follows directly from the equality
    \begin{equation}\label{eq_Pacific_proof_T_theta(N)=0<->T_theta_N(N)=0}
    {\text{\Fontlukas T}}_{\Theta_{\mathcal{I}}}{\mathcal{I}}(X,Y)={\text{\Fontlukas T}}_{\Theta}{\mathcal{I}}({\mathcal{I}}X,Y)+{\text{\Fontlukas T}}_{\Theta}{\mathcal{I}}(X,{\mathcal{I}}Y)-{\mathcal{I}}\left({\text{\Fontlukas T}}_{\Theta}{\mathcal{I}}(X,Y)\right),
    \end{equation}
    that holds for all sections $X,Y$ of $E$ (see \cite{ALC11}). Let us now suppose that ${\text{\Fontlukas T}}_{\Theta_{\mathcal{I}}}{\mathcal{I}}=0$. Evaluating (\ref{eq_Pacific_proof_T_theta(N)=0<->T_theta_N(N)=0}) on pairs of sections of $E$ of type $({\mathcal{I}}X,Y)$ and of type $(X,{\mathcal{I}}Y)$, we obtain
    \begin{equation}
        {\mathcal{I}}\left({\text{\Fontlukas T}}_{\Theta}{\mathcal{I}}({\mathcal{I}}X,Y)\right)
            =\lambda\,{\text{\Fontlukas T}}_{\Theta}{\mathcal{I}}(X,Y)
            +{\text{\Fontlukas T}}_{\Theta}{\mathcal{I}}({\mathcal{I}}X,{\mathcal{I}}Y)\label{eq_aux1_proof_T_theta(N)=0<->T_theta_N(N)=0}
    \end{equation}
and
    \begin{equation}
        {\mathcal{I}}\left({\text{\Fontlukas T}}_{\Theta}{\mathcal{I}}(X,{\mathcal{I}}Y)\right)=
            \lambda\,{\text{\Fontlukas T}}_{\Theta}{\mathcal{I}}(X,Y)
            + {\text{\Fontlukas T}}_{\Theta}{\mathcal{I}}({\mathcal{I}}X,{\mathcal{I}}Y),\label{eq_aux2_proof_T_theta(N)=0<->T_theta_N(N)=0}
    \end{equation}
respectively. Equations (\ref{eq_aux1_proof_T_theta(N)=0<->T_theta_N(N)=0}) and (\ref{eq_aux2_proof_T_theta(N)=0<->T_theta_N(N)=0}) yield
        $${\mathcal{I}}\left({\text{\Fontlukas T}}_{\Theta}{\mathcal{I}}({\mathcal{I}}X,Y)\right)=
            {\mathcal{I}}\left({\text{\Fontlukas T}}_{\Theta}{\mathcal{I}}(X,{\mathcal{I}}Y)\right),$$
which implies
    \begin{equation}\label{eq_aux3_proof_T_theta(N)=0<->T_theta_N(N)=0}
        {\text{\Fontlukas T}}_{\Theta}{\mathcal{I}}({\mathcal{I}}X,Y)={\text{\Fontlukas T}}_{\Theta}{\mathcal{I}}(X,{\mathcal{I}}Y).
    \end{equation}
    Substituting (\ref{eq_aux3_proof_T_theta(N)=0<->T_theta_N(N)=0}) in (\ref{eq_Pacific_proof_T_theta(N)=0<->T_theta_N(N)=0}), we get
    \begin{equation}\label{eq_aux4_proof_T_theta(N)=0<->T_theta_N(N)=0}
        2\, {\text{\Fontlukas T}}_{\Theta}{\mathcal{I}}({\mathcal{I}}X,Y)={\mathcal{I}}\left({\text{\Fontlukas T}}_{\Theta}{\mathcal{I}}(X,Y)\right).
    \end{equation}
    Applying ${\mathcal{I}}$ to both sides of (\ref{eq_aux4_proof_T_theta(N)=0<->T_theta_N(N)=0}), we get
    \begin{equation}\label{eq_aux5_proof_T_theta(N)=0<->T_theta_N(N)=0}
       2 {\mathcal{I}}\left({\text{\Fontlukas T}}_{\Theta}{\mathcal{I}}({\mathcal{I}}X,Y)\right)=\lambda\, {\text{\Fontlukas T}}_{\Theta}{\mathcal{I}}(X,Y).
    \end{equation}
    On the other hand, replacing $X$ by $\mathcal{I}X$ in (\ref{eq_aux4_proof_T_theta(N)=0<->T_theta_N(N)=0}), we obtain
    \begin{equation}\label{eq_aux6_proof_T_theta(N)=0<->T_theta_N(N)=0}
       2\lambda\, {\text{\Fontlukas T}}_{\Theta}{\mathcal{I}}(X,Y)= {\mathcal{I}}\left({\text{\Fontlukas T}}_{\Theta}{\mathcal{I}}({\mathcal{I}}X,Y)\right).
    \end{equation}
    From (\ref{eq_aux5_proof_T_theta(N)=0<->T_theta_N(N)=0}) and (\ref{eq_aux6_proof_T_theta(N)=0<->T_theta_N(N)=0}) we conclude that ${\text{\Fontlukas T}}_{\Theta}{\mathcal{I}}=0$, as pretended.
\end{proof}

\begin{prop}\label{prop_T_theta(J)=0<->T_theta_I(J)=0}
    Let $\mathcal{I}$ and $\mathcal{J}$ be two anti-commuting endomorphisms on a pre-Courant algebroid $(E, \Theta)$, such that $C_{\Theta}(\mathcal{I},\mathcal{J})=0$ and $\mathcal{I}^2=\lambda\,{\rm id}_E$, for some $\lambda \in \mathbb{R}\!\setminus\!\{0\}$. Then, ${\text{\Fontlukas T}}_{\Theta}\mathcal{J}=0 \Leftrightarrow {\text{\Fontlukas T}}_{\Theta_\mathcal{I}}\mathcal{J}=0$.
\end{prop}
\begin{proof}
    The statement \mbox{${\text{\Fontlukas T}}_{\Theta}\mathcal{J}=0 \Rightarrow {\text{\Fontlukas T}}_{\Theta_\mathcal{I}}\mathcal{J}=0$} is an immediate consequence of Proposition $4.18$ in \cite{ALC11}. In fact, when $C_{\Theta}(\mathcal{I},\mathcal{J})=0$, that proposition gives
    \begin{equation}\label{eq_Pacific_proof_T_theta(J)=0<->T_theta_I(J)=0}
    {\text{\Fontlukas T}}_{\Theta_\mathcal{I}}\mathcal{J}(X,Y)=-{\text{\Fontlukas T}}_{\Theta}\mathcal{J}(IX,Y)-{\text{\Fontlukas T}}_{\Theta}\mathcal{J}(X,\mathcal{I}Y)-\mathcal{I}\left({\text{\Fontlukas T}}_{\Theta}\mathcal{J}(X,Y)\right),
    \end{equation}
    for all sections $X,Y$ of $E$. Let us now suppose that ${\text{\Fontlukas T}}_{\Theta_\mathcal{I}}\mathcal{J}=0$.
Evaluating (\ref{eq_Pacific_proof_T_theta(J)=0<->T_theta_I(J)=0}) on pairs of sections of $E$ of type $(\mathcal{I}X,Y)$ and of type $(X,\mathcal{I}Y)$ and doing similar computations as in the proof of Proposition~\ref{prop_T_theta(N)=0<->T_theta_N(N)=0}, we obtain
    \begin{equation}\label{eq_aux3_proof_T_theta(J)=0<->T_theta_I(J)=0}
        {\text{\Fontlukas T}}_{\Theta}\mathcal{J}(\mathcal{I}X,Y)={\text{\Fontlukas T}}_{\Theta}\mathcal{J}(X,\mathcal{I}Y).
    \end{equation}
    Substituting (\ref{eq_aux3_proof_T_theta(J)=0<->T_theta_I(J)=0}) in (\ref{eq_Pacific_proof_T_theta(J)=0<->T_theta_I(J)=0}), we get
    \begin{equation}\label{eq_aux4_proof_T_theta(J)=0<->T_theta_I(J)=0}
        2\, {\text{\Fontlukas T}}_{\Theta}\mathcal{J}(\mathcal{I}X,Y)=-\mathcal{I}\left({\text{\Fontlukas T}}_{\Theta}\mathcal{J}(X,Y)\right).
    \end{equation}
    Applying $\mathcal{I}$ to both sides of (\ref{eq_aux4_proof_T_theta(J)=0<->T_theta_I(J)=0}), yields
    \begin{equation}\label{eq_aux5_proof_T_theta(J)=0<->T_theta_I(J)=0}
       2 \mathcal{I}\left({\text{\Fontlukas T}}_{\Theta}\mathcal{J}(\mathcal{I}X,Y)\right)=\lambda\, {\text{\Fontlukas T}}_{\Theta}\mathcal{J}(X,Y).
    \end{equation}
    On the other hand, replacing $X$ by $\mathcal{I}X$ in (\ref{eq_aux4_proof_T_theta(J)=0<->T_theta_I(J)=0}), we obtain
    \begin{equation}\label{eq_aux6_proof_T_theta(J)=0<->T_theta_I(J)=0}
       2\lambda\, {\text{\Fontlukas T}}_{\Theta}\mathcal{J}(X,Y)= \mathcal{I}\left({\text{\Fontlukas T}}_{\Theta}\mathcal{J}(\mathcal{I}X,Y)\right).
    \end{equation}
    From (\ref{eq_aux5_proof_T_theta(J)=0<->T_theta_I(J)=0}) and (\ref{eq_aux6_proof_T_theta(J)=0<->T_theta_I(J)=0}) we conclude that ${\text{\Fontlukas T}}_{\Theta}\mathcal{J}=0$, as pretended.
\end{proof}

\begin{rem}
    Notice that although the statement of Proposition $4.18$ in \cite{ALC11} refers to skew-symmetric endomorphisms, the skew-symmetry is not used in the proof, so that we could use (\ref{eq_Pacific_proof_T_theta(J)=0<->T_theta_I(J)=0}) in the proof of Proposition~\ref{prop_T_theta(J)=0<->T_theta_I(J)=0}.
\end{rem}


\begin{prop} \label{C_theta(I,J)=0}
    Let $\mathcal{I}$ and $\mathcal{J}$ be two anti-commuting endomorphisms on a pre-Courant algebroid $(E, \Theta)$, with vanishing Nijenhuis torsion. Then,
    \begin{enumerate}
    \item $C_{\Theta}(\mathcal{I},\mathcal{I}\mathcal{J})(X,Y)=\mathcal{I}\big(C_{\Theta}(\mathcal{I},\mathcal{J})(X,Y)\big)$,
    for all sections $X, Y \in \Gamma(E)$.
    \end{enumerate}
   Moreover, if $\mathcal{I}^2=\lambda_\mathcal{I} \,{\rm id}_E$ and $\mathcal{J}^2=\lambda_\mathcal{J} \,{\rm id}_E$, for some $\lambda_\mathcal{I}, \lambda_\mathcal{J} \in \mathbb{R}\backslash\{0\}$, then
    \begin{enumerate}
    \setcounter{enumi}{1}
    \item $C_{\Theta}(\mathcal{I},\mathcal{J})=0.$
    \end{enumerate}
\end{prop}

\begin{proof}
\begin{enumerate}
\item It is directly obtained from Proposition $3.13$ in \cite{ALC11}.


\item    From $\mathcal{I}^2=\lambda_\mathcal{I} \,{\rm id}_E$ we have $\mathcal{J}=\lambda_\mathcal{I}^{-1} \mathcal{I}(\mathcal{I}\mathcal{J})$ and, for all $X,Y \in \Gamma(E)$,
    \begin{equation} \label{conc_J2}
    C_{\Theta}(\mathcal{I},\mathcal{J})(X,Y)=\lambda_\mathcal{I}^{-1} C_{\Theta}(\mathcal{I},\mathcal{I}(\mathcal{I}\mathcal{J}))(X,Y)=\lambda_\mathcal{I}^{-1} \mathcal{I}(C_{\Theta}(\mathcal{I}\mathcal{J},\mathcal{I})(X,Y)),\end{equation}
    where, in the last equality, we used $(i)$ and the fact that $C_{\Theta}(\cdot,\cdot)$ is symmetric.
Using $\mathcal{J}^2=\lambda_\mathcal{J} \,{\rm id}_E$ and $(i)$, Equation (\ref{conc_J2}) becomes
    \begin{multline*}
        C_{\Theta}(\mathcal{I},\mathcal{J})(X,Y)=\lambda_\mathcal{I}^{-1}\lambda_\mathcal{J}^{-1} \mathcal{I}(C_{\Theta}(\mathcal{I}\mathcal{J},\mathcal{I}\mathcal{J}(\mathcal{J}))(X,Y))\\
        =\lambda_\mathcal{I}^{-1}\lambda_\mathcal{J}^{-1} \mathcal{I}(\mathcal{I}\mathcal{J})(C_{\Theta}(\mathcal{I}\mathcal{J},\mathcal{J})(X,Y))=-\lambda_\mathcal{J}^{-1} \mathcal{J}(C_{\Theta}(\mathcal{J}\mathcal{I},\mathcal{J})(X,Y)),
    \end{multline*}
    where we used the fact that $\mathcal{I}\mathcal{J}=-\mathcal{J}\mathcal{I}$ in the last equality. Finally, applying $(i)$ once more, we get
    $$C_{\Theta}(\mathcal{I},\mathcal{J})(X,Y)=-\lambda_\mathcal{J}^{-1} \mathcal{J}^2(C_{\Theta}(\mathcal{I},\mathcal{J})(X,Y))=-C_{\Theta}(\mathcal{I},\mathcal{J})(X,Y).$$
    Therefore,
    $$C_{\Theta}(\mathcal{I},\mathcal{J})(X,Y)=0,$$
    for all $X,Y \in \Gamma(E)$.
    \end{enumerate}
\end{proof}


\section{Hypersymplectic structures on pre-Courant algebroids} \label{section_2}

In this section we introduce the notion of an $\boldsymbol{\varepsilon}$-\emph{hypersymplectic structure} on a pre-Courant algebroid $(E, \Theta)$ and study the main relations and properties of the induced morphisms.

Along the paper, in order to simplify the notation, when $\mathcal{I}$ and $\mathcal{J}$ are endomorphisms of $E$, the composition $\mathcal{I}\circ \mathcal{J}$ will be denoted by $\mathcal{I}\mathcal{J}$. Also, we consider $1, 2$ and $3$ as the representative elements of the equivalence classes of $\mathbb{Z}_3$, i.e., $\mathbb{Z}_3:=\{[1],[2],[3]\}$. Although we omit the brackets, and write $i$ instead of $[i]$, the indices (and corresponding computations) must be thought in $\mathbb{Z}_3:=\mathbb{Z}/{3\mathbb{Z}}$.

\begin{defn}\label{def_epsilonHS_pre-Courant}
An $\boldsymbol{\varepsilon}$-\emph{hypersymplectic structure} on a pre-Courant algebroid $(E, \Theta)$ is a triplet $(\mathcal{S}_1, \mathcal{S}_2, \mathcal{S}_3)$ of skew-symmetric endomorphisms $\mathcal{S}_i: E \to E$, $i=1,2,3$, such that
\begin{enumerate}
  \item ${\mathcal{S}_i}^2=\varepsilon_i \,{\rm id}_E$,
  \item $\mathcal{S}_i\mathcal{S}_j=\varepsilon_1\varepsilon_2\varepsilon_3\mathcal{S}_j\mathcal{S}_i$,\ \ $i\neq j \in \{1,2,3\}$,
  \item $\Theta_{\mathcal{S}_i,\mathcal{S}_i}=\varepsilon_i \Theta$,
\end{enumerate}
where the parameters $\varepsilon_i=\pm 1$ form the triplet $\boldsymbol{\varepsilon}=(\varepsilon_1,\varepsilon_2,\varepsilon_3)$.
\end{defn}

From conditions i) and iii) of Definition~\ref{def_epsilonHS_pre-Courant}, and using formula (\ref{supergeometric_torsion}), we immediately have the following proposition.

\begin{prop}  \label{Si_Nijenhuis}
Let $(\mathcal{S}_1, \mathcal{S}_2, \mathcal{S}_3)$ be an $\boldsymbol{\varepsilon}$-hypersymplectic structure on a pre-Courant algebroid $(E, \Theta)$. Then, $\mathcal{S}_1, \mathcal{S}_2$ and $\mathcal{S}_3$ are Nijenhuis morphisms.
\end{prop}

\begin{rem}\label{rem_equiv_cond_def_epsilonHS_pre-Courant}
    Notice that according to Proposition \ref{Si_Nijenhuis}, condition iii) in Definition \ref{def_epsilonHS_pre-Courant} can be replaced by\\
    \indent iii') ${\text{\Fontlukas T}}_{\Theta}\mathcal{S}_i=0, \quad i=1,2,3.$
\end{rem}

Given an $\boldsymbol{\varepsilon}$-hypersymplectic structure $(\mathcal{S}_1, \mathcal{S}_2, \mathcal{S}_3)$ on $(E, \Theta)$, let us define the morphisms $\mathcal{T}_1,\mathcal{T}_2$ and $\mathcal{T}_3$ by setting
\begin{equation} \label{def_transition}
\mathcal{T}_i:=\varepsilon_{i-1} \mathcal{S}_{i-1}\mathcal{S}_{i+1},
\end{equation}
where the indices must be considered as elements of $\mathbb{Z}_3$.
%

    The morphisms $\mathcal{T}_i,\, i=1,2,3,$ can be seen as transition maps between the morphisms $\mathcal{S}_j, j=1,2,3$. In fact we have, for all $i \in \mathbb{Z}_3$,
\begin{equation*}
\mathcal{S}_{i-1} \mathcal{T}_i = \mathcal{S}_{i+1}.
\end{equation*}

The picture in Figure \ref{Fig. 1} is a good way to visualize these relations. For example, the bottom triangle shows that $\mathcal{S}_{2} \mathcal{T}_3= \mathcal{S}_{1}$ and $\varepsilon_1 \mathcal{T}_3\mathcal{S}_{1}= \varepsilon_2\mathcal{S}_{2}$. For the latter equality we use the fact that the inverse of morphism $\mathcal{S}_{i}$ is $\varepsilon_i\mathcal{S}_{i}$.

\begin{figure}[H]
  \begin{pspicture}(-1,-0.6)(5,3.5)
        \psline(0,0)(4,0)
        \psline(4,0)(2,3.46)
        \psline(2,3.46)(0,0)
        \psline(2,1.15)(0,0)
        \psline(2,3.46)(2,1.15)
        \psline(2,1.15)(4,0)
        \begin{large}
        \rput[tl](1.86,0.15){$\boldsymbol <$}
        \rput[tl]{300}(3.06,1.9){$\boldsymbol >$}
        \rput[tl]{60}(0.77,1.6){$\boldsymbol >$}
        \end{large}
        \rput[tl]{90}(1.88,1.9){$\boldsymbol <$}
        \rput[tl]{30}(1.0,0.72){$\boldsymbol >$}
        \rput[tl]{330}(2.75,0.85){$\boldsymbol <$}
        \rput[tl]{330}(2.55,0.6){$\mathcal{S}_1$}
        \rput[tl]{270}(2.45,2.2){$\mathcal{S}_3$}
        \rput[tl]{30}(0.85,1.0){$\mathcal{S}_2$}
        \rput[tl]{300}(3.36,2.04){$\mathcal{T}_2$}
        \rput[tl]{60}(0.49,1.8){$\mathcal{T}_1$}
        \rput[tl](1.86,-0.2){$\mathcal{T}_3$}
        \end{pspicture}
        \caption{}\label{Fig. 1}
\end{figure}

\begin{prop}\label{propriedades_Ti}
Let $(\mathcal{S}_1, \mathcal{S}_2, \mathcal{S}_3)$ be an $\boldsymbol{\varepsilon}$-hypersymplectic structure on a pre-Courant algebroid $(E, \Theta)$. The morphisms $\mathcal{T}_1$, $\mathcal{T}_2$ and $\mathcal{T}_3$ satisfy the following relations for all $i=1,2,3$:
\begin{enumerate}
  \item ${\mathcal{T}_i}^*=\varepsilon_1\varepsilon_2\varepsilon_3 \mathcal{T}_i$;
  \item ${\mathcal{T}_i}^2=\varepsilon_i \,{\rm id}_E$;
  \item $\mathcal{T}_{i-1}\mathcal{T}_{i+1}=\varepsilon_1\varepsilon_2\varepsilon_3\mathcal{T}_{i+1}\mathcal{T}_{i-1}=\varepsilon_i \mathcal{T}_i$;
  \item $\mathcal{T}_3\mathcal{T}_2\mathcal{T}_1=\varepsilon_1\varepsilon_2\varepsilon_3 \mathcal{T}_1\mathcal{T}_2\mathcal{T}_3={\rm id}_E$.
\end{enumerate}
\end{prop}
\begin{proof}
Using conditions i) and ii) of Definition \ref{def_epsilonHS_pre-Courant} and also equation (\ref{def_transition}), we have:
\begin{enumerate}
 \item ${\mathcal{T}_i}^*= (\varepsilon_{i-1} \mathcal{S}_{i-1}\mathcal{S}_{i+1})^*=\varepsilon_{i-1} \mathcal{S}_{i+1}\mathcal{S}_{i-1}=\varepsilon_i\varepsilon_{i+1}\mathcal{S}_{i-1}\mathcal{S}_{i+1}=\varepsilon_1\varepsilon_2\varepsilon_3 \mathcal{T}_i,$\\
 where we also used the fact that the endomorphisms $\mathcal{S}_i$ are skew-symmetric;

 \item ${\mathcal{T}_i}^2= \mathcal{S}_{i-1}\mathcal{S}_{i+1}\mathcal{S}_{i-1}\mathcal{S}_{i+1}= \varepsilon_1\varepsilon_2\varepsilon_3\mathcal{S}_{i-1}^2\mathcal{S}_{i+1}^2 = \varepsilon_i \,{\rm id}_E;$

 \item $\mathcal{T}_{i-1}\mathcal{T}_{i+1}=  \varepsilon_{i+1}\varepsilon_{i} \mathcal{S}_{i+1}\mathcal{S}_{i}\mathcal{S}_{i}\mathcal{S}_{i-1}= \varepsilon_{i+1}\mathcal{S}_{i+1}\mathcal{S}_{i-1}=\varepsilon_{i}\varepsilon_{i-1}\mathcal{S}_{i-1}\mathcal{S}_{i+1}=\varepsilon_i \mathcal{T}_i.$\\
 This proves one part of the statement and we use it to prove the second equality of the statement. In fact, from $(\mathcal{T}_{i-1}\mathcal{T}_{i+1})^2=(\varepsilon_i \mathcal{T}_i)^2=\varepsilon_i  \,{\rm id}_E$  and using item ii),
 we have
 $$\mathcal{T}_{i-1}\mathcal{T}_{i+1}=\varepsilon_i (\mathcal{T}_{i-1}\mathcal{T}_{i+1})^{-1}=\varepsilon_i (\mathcal{T}_{i+1})^{-1}(\mathcal{T}_{i-1})^{-1}=\varepsilon_1\varepsilon_2\varepsilon_3\mathcal{T}_{i+1}\mathcal{T}_{i-1}.$$
    \item By item iii), $\mathcal{T}_3\mathcal{T}_2=\varepsilon_1 \mathcal{T}_{1}$; then, using item ii),
    $$\mathcal{T}_3\mathcal{T}_2\mathcal{T}_1=\varepsilon_1 \mathcal{T}_{1}^2={\rm id}_E.$$
    Furthermore, using item iii) three times, we can change the order of $\mathcal{T}_i$'s in the product $\mathcal{T}_3\mathcal{T}_2\mathcal{T}_1$ to get
    $$\mathcal{T}_3\mathcal{T}_2\mathcal{T}_1=(\varepsilon_1\varepsilon_2\varepsilon_3)^3\mathcal{T}_1\mathcal{T}_2\mathcal{T}_3= \varepsilon_1\varepsilon_2\varepsilon_3 \mathcal{T}_1\mathcal{T}_2\mathcal{T}_3.$$
\end{enumerate}
\end{proof}

\begin{rem}  \label{stienon}
In the particular case where $\varepsilon_1=\varepsilon_2=\varepsilon_3=-1$ and $(E, \Theta)$ is a Courant algebroid, the triplet $(\mathcal{T}_{1}, \mathcal{T}_{2}, \mathcal{T}_{3})$ is an \emph{almost hypercomplex structure} on $(E, \Theta)$ in the terminology of \cite{Stienon}.
\end{rem}

Given an $\boldsymbol{\varepsilon}$-hypersymplectic structure $(\mathcal{S}_1, \mathcal{S}_2, \mathcal{S}_3)$ on a pre-Courant algebroid $(E, \Theta)$, we may define an endomorphism
$\mathcal{G}:E \to E$  by setting, for all $i=1,2,3$,
\begin{equation}  \label{def_mathcalG}
\mathcal{G}:=\mathcal{S}_{i+1}\mathcal{S}_{i}\mathcal{S}_{i-1}.
\end{equation}
Notice that $\mathcal{G}$ is well defined by (\ref{def_mathcalG}). In fact, since $\mathcal{S}_i\mathcal{S}_j=\varepsilon_1\varepsilon_2\varepsilon_3\mathcal{S}_j\mathcal{S}_i$, for $i\neq j$, we obviously have $\mathcal{G}=\mathcal{S}_3\mathcal{S}_2\mathcal{S}_1=\mathcal{S}_1\mathcal{S}_3\mathcal{S}_2=\mathcal{S}_2\mathcal{S}_1\mathcal{S}_3$.

\begin{prop}  \label{propriedades_G}
    Let $(\mathcal{S}_1, \mathcal{S}_2, \mathcal{S}_3)$ be an $\boldsymbol{\varepsilon}$-hypersymplectic structure on a pre-Courant algebroid $(E, \Theta)$. Then, the morphism $\mathcal{G}$, given by (\ref{def_mathcalG}), satisfies the following properties:
    \begin{enumerate}
        \item $\mathcal{G}^*=-\varepsilon_1\varepsilon_2\varepsilon_3 \mathcal{G}$;
        \item $\mathcal{G}^2={\rm id}_E$.
    \end{enumerate}
\end{prop}
\begin{proof}
\begin{enumerate}
\item An easy computation using the skew-symmetry of each $\mathcal{S}_i$ and condition ii) in Definition \ref{def_epsilonHS_pre-Courant}, gives
\begin{equation*}
\mathcal{G}^*=(\mathcal{S}_{i+1}\mathcal{S}_{i}\mathcal{S}_{i-1})^*=-\mathcal{S}_{i-1}\mathcal{S}_{i}\mathcal{S}_{i+1}=-\varepsilon_1\varepsilon_2\varepsilon_3 \mathcal{S}_{i-1}\mathcal{S}_{i+1}\mathcal{S}_{i}=-\varepsilon_1\varepsilon_2\varepsilon_3 \mathcal{G}.
\end{equation*}
\item The proof is immediate using properties of $\mathcal{S}_i$ from Definition \ref{def_epsilonHS_pre-Courant}:
$$\mathcal{G}^2=(\mathcal{S}_3\mathcal{S}_2\mathcal{S}_1)^2=\varepsilon_1\varepsilon_2\varepsilon_3\, \mathcal{S}_3^2\mathcal{S}_2^2\mathcal{S}_1^2
= {\rm id}_E.$$
\end{enumerate}

\end{proof}

The next proposition shows that, for each $i$, the morphisms $\mathcal{G}$, $\mathcal{S}_i$ and $\mathcal{T}_i$ commute pairwise and each one is obtained out of the other two.

\begin{prop}\label{propriedades_Ti_Si_G}
Let $(\mathcal{S}_1, \mathcal{S}_2, \mathcal{S}_3)$ be an $\boldsymbol{\varepsilon}$-hypersymplectic structure on a pre-Courant algebroid $(E, \Theta)$. The morphisms $\mathcal{S}_i, \mathcal{T}_i$ and $\mathcal{G}, i=1,2,3,$ satisfy the following relations:
\begin{enumerate}
    \item $\mathcal{T}_i\mathcal{S}_i=\mathcal{S}_i\mathcal{T}_i=\varepsilon_{i-1}\mathcal{G};$
    \item $\mathcal{G}\mathcal{S}_i=\mathcal{S}_i\mathcal{G}=\varepsilon_{i-1}\varepsilon_{i}\mathcal{T}_i;$
    \item $\mathcal{G}\mathcal{T}_i=\mathcal{T}_i\mathcal{G}=\varepsilon_{i-1}\varepsilon_{i}\mathcal{S}_i.$
\end{enumerate}
Moreover, for all $i\neq j \in \{1,2,3\}$,
\begin{enumerate}
\setcounter{enumi}{3}
    \item $\mathcal{S}_j \mathcal{T}_i=\varepsilon_1\varepsilon_2\varepsilon_3 \mathcal{T}_i \mathcal{S}_j=\left\{
                                                                                                             \begin{array}{ll}
                                                                                                               \mathcal{S}_{i+1}, & j=i-1 \\
                                                                                                               \varepsilon_{i}\mathcal{S}_{i-1}, & j=i+1.
                                                                                                             \end{array}
                                                                                                           \right.
    $
\end{enumerate}
\end{prop}
\begin{proof}
\begin{enumerate}
    \item Using (\ref{def_transition}) and the condition ii) of Definition \ref{def_epsilonHS_pre-Courant} twice, we get
    $$\mathcal{T}_i\mathcal{S}_i=\varepsilon_{i-1} \mathcal{S}_{i-1}\mathcal{S}_{i+1}\mathcal{S}_i =
     \varepsilon_{i-1} \mathcal{S}_i\mathcal{S}_{i-1}\mathcal{S}_{i+1} =\mathcal{S}_i\mathcal{T}_i.$$
    On the other hand, from (\ref{def_transition}) and (\ref{def_mathcalG}) we have
    $$\mathcal{T}_i\mathcal{S}_i=\varepsilon_{i-1} \mathcal{S}_{i-1}\mathcal{S}_{i+1}\mathcal{S}_i =\varepsilon_{i-1}\mathcal{G}.$$
    \item From item i) we have $\mathcal{T}_i\mathcal{S}_i=\varepsilon_{i-1}\mathcal{G}$ and composing with $\mathcal{S}_i$, on the right, we get  $\mathcal{T}_i(\mathcal{S}_i)^2=\varepsilon_{i-1}\mathcal{G}\mathcal{S}_i$ or, equivalently, $\varepsilon_{i-1}\varepsilon_{i}\mathcal{T}_i=\mathcal{G}\mathcal{S}_i$. For the other equality, we start with $\mathcal{S}_i\mathcal{T}_i=\varepsilon_{i-1}\mathcal{G}$ and compose with $\mathcal{S}_i$, on the left, to obtain $(\mathcal{S}_i)^2\mathcal{T}_i=\varepsilon_{i-1}\mathcal{S}_i\mathcal{G}$; so that $\varepsilon_{i-1}\varepsilon_{i}\mathcal{T}_i=\mathcal{S}_i\mathcal{G}$.
    \item Analogous to the proof of item ii), but composing with $\mathcal{T}_i$ instead of $\mathcal{S}_i$.
    \item Let us prove the case $j=i-1$. Using (\ref{def_transition}) and the condition i) of Definition \ref{def_epsilonHS_pre-Courant}, we have
    $$\mathcal{S}_{i-1} \mathcal{T}_i=\varepsilon_{i-1}\mathcal{S}_{i-1}^2\mathcal{S}_{i+1}=\mathcal{S}_{i+1}.$$
    Moreover, by (\ref{def_transition}) and conditions i) and ii) of Definition \ref{def_epsilonHS_pre-Courant}, we get
    $$\mathcal{T}_i\mathcal{S}_{i-1} =\varepsilon_{i-1}\mathcal{S}_{i-1}\mathcal{S}_{i+1}\mathcal{S}_{i-1}=\varepsilon_{i}\varepsilon_{i+1}\mathcal{S}_{i-1}^2\mathcal{S}_{i+1}=\varepsilon_1\varepsilon_2\varepsilon_3\mathcal{S}_{i+1},$$
    which completes the proof of the statement. The case $j=i+1$ is analogous.
\end{enumerate}
\end{proof}

Notice that when $\mathcal{G}={\rm id}_E$, then $\mathcal{S}_i=\varepsilon_{i-1}\varepsilon_{i}\mathcal{T}_i$. In this case, if\,\footnote{Notice that, because of Proposition \ref{propriedades_G} i), the assumption $\mathcal{G}={\rm id}_E$ implies $\varepsilon_1\varepsilon_2\varepsilon_3=-1$.} $\varepsilon_1=\varepsilon_2=\varepsilon_3=-1$ and $\Theta$ is a Courant algebroid structure on $E$, the triplet $(\mathcal{S}_{1}, \mathcal{S}_{2}, \mathcal{S}_{3})$ is an \emph{hypercomplex structure} on $(E, \Theta)$, in the sense of \cite{Stienon}.

\

The relations between $\mathcal{S}_i, \mathcal{T}_j$ and $\mathcal{G}$, for all $i,j=1,2,3$, may be visualized in Figure \ref{Fig. 2}.

\begin{figure}[H]
\begin{center}
\psscalebox{0.7}{
\begin{pspicture*}(-2.75,-1.16)(17,11.26)
\psline[linewidth=1.6pt](3,5.2)(6,0)
\psline[linewidth=1.6pt](0,0)(6,0)
\psline[linewidth=1.6pt](0,0)(3,5.2)
\psline(0,0)(3,1.73)
\psline(3,5.2)(3,1.73)
\psline(6,0)(3,1.73)
\psline(9,5.2)(6,3.46)
\psline[linewidth=1.6pt](9,5.2)(3,5.2)
\psline(3,5.2)(6,3.46)
\psline(6,0)(6,3.46)
\psline[linewidth=1.6pt](9,5.2)(6,0)
\psline(3,5.2)(6,6.93)
\psline(9,5.2)(6,6.93)
\psline[linewidth=1.6pt](9,5.2)(6,10.39)
\psline(6,10.39)(6,6.93)
\psline[linewidth=1.6pt](3,5.2)(6,10.39)
\psline[linewidth=1.6pt](9,5.2)(12,0)
\psline(12,0)(9,1.73)
\psline[linewidth=1.6pt](12,0)(6,0)
\psline(6,0)(9,1.73)
\psline(9,5.2)(9,1.73)
\begin{huge}
\rput[tl](5.66,5.4){$\boldsymbol >$}
\rput[tl]{60}(7.37,2.8){$\boldsymbol <$}
\rput[tl]{300}(4.46,3.1){$\boldsymbol <$}
\rput[tl](2.68,0.2){$\boldsymbol <$}
\rput[tl](8.68,0.2){$\boldsymbol >$}
\rput[tl]{60}(1.245,2.58){$\boldsymbol <$}
\rput[tl]{60}(4.15,7.61){$\boldsymbol >$}
\rput[tl]{300}(10.45,3.08){$\boldsymbol >$}
\rput[tl]{300}(7.5,8.2){$\boldsymbol <$}
\end{huge}
\begin{large}
\rput[tl](5.7,4.9){$\mathcal{T}_1 $}
\rput[tl]{60}(7,3){$\mathcal{T}_2 $}
\rput[tl]{300}(5.0,3.2){$\mathcal{T}_3 $}
\rput[tl](8.83,-0.25){$\mathcal{S}_1$}
\rput[tl](2.72,-0.25){$\mathcal{S}_1$}
\rput[tl]{60}(0.9,2.9){$\mathcal{S}_2$}
\rput[tl]{60}(3.9,7.95){$\mathcal{S}_2$}
\rput[tl]{300}(10.9,3.1){$\mathcal{S}_3$}
\rput[tl]{300}(8,8.3){$\mathcal{S}_3$}
\end{large}
\rput[tl]{330}(4.68,4.36){$\boldsymbol <$}
\rput[tl]{270}(6.12,2.2){$\boldsymbol >$}
\rput[tl]{30}(7.12,4.25){$\boldsymbol >$}
\rput[tl]{30}(1.65,1.08){$\boldsymbol <$}
\rput[tl]{90}(2.88,2.9){$\boldsymbol >$}
\rput[tl]{330}(4.15,1.2){$\boldsymbol >$}
\rput[tl]{30}(4.65,6.28){$\boldsymbol <$}
\rput[tl]{90}(5.88,8.1){$\boldsymbol >$}
\rput[tl]{330}(7.15,6.4){$\boldsymbol >$}
\rput[tl]{30}(7.65,1.08){$\boldsymbol <$}
\rput[tl]{90}(8.88,2.9){$\boldsymbol >$}
\rput[tl]{330}(10.15,1.2){$\boldsymbol >$}
\begin{small}
\rput[tl]{330}(4.32,4.14){$\varepsilon_1 \mathcal{T}_2 $}
\rput[tl]{270}(6.45,2.4){$\varepsilon_3 \mathcal{T}_1 $}
\rput[tl]{30}(6.9,4.55){$\varepsilon_2 \mathcal{T}_3 $}
\rput[tl]{30}(1.22,1.35){$\varepsilon_3\varepsilon_2 \mathcal{T}_3$}
\rput[tl]{90}(3.25,2.8){$\varepsilon_3 \mathcal{S}_1$}
\rput[tl]{330}(3.55,1.1){$\varepsilon_2\varepsilon_1 \mathcal{S}_2$}
\rput[tl]{30}(4.22,6.55){$\varepsilon_3\varepsilon_2 \mathcal{S}_3$}
\rput[tl]{90}(6.25,8){$\varepsilon_1\varepsilon_3 \mathcal{T}_1$}
\rput[tl]{330}(6.55,6.3){$\varepsilon_1 \mathcal{S}_2$}
\rput[tl]{30}(7.22,1.35){$\varepsilon_2 \mathcal{S}_3$}
\rput[tl]{90}(9.25,2.8){$\varepsilon_1\varepsilon_3 \mathcal{S}_1$}
\rput[tl]{330}(9.55,1.1){$\varepsilon_2\varepsilon_1 \mathcal{T}_2$}
\end{small}
\begin{large}
\rput(-0.5,-0.25){$\boldsymbol D$}
\rput(6,10.85){$\boldsymbol D$}
\rput(12.4,-0.25){$\boldsymbol D$}
\rput(6,-0.5){$\boldsymbol A$}
\rput(2.5,5.3){$\boldsymbol B$}
\rput(9.4,5.3){$\boldsymbol C$}
\end{large}
\end{pspicture*}
}
\end{center}
\caption{}\label{Fig. 2}
\end{figure}

This is to be understood as the pattern for a tetrahedron $ABCD$. The metric $\mathcal{G}$ does not appear in Figure \ref{Fig. 2} but in Figure \ref{Fig. 3}, after building the tetrahedron, $\mathcal{G}$ appears as the altitude of the tetrahedron $ABCD$.

\begin{figure}[H]
\psscalebox{0.8}{
\begin{pspicture*}(-2,-3.5)(7,3.5)
\psline[linewidth=1.6pt,linestyle=dashed,dash=5pt 5pt](-1.12,-0.84)(4.08,-0.06)
\psline[linestyle=dashed,dash=5pt 5pt](1.82,-1.26)(-1.12,-0.84)
\psline[linestyle=dashed,dash=5pt 5pt](1.82,-1.26)(2.5,-2.88)
\psline[linestyle=dashed,dash=5pt 5pt](1.82,-1.26)(4.08,-0.06)
\psline[linewidth=1.6pt,linestyle=dashed,dash=5pt 5pt,linecolor=red](1.82,2.8)(1.82,-1.26)
\psline[linewidth=1.6pt](1.82,2.8)(-1.12,-0.84)
\psline[linewidth=1.6pt](-1.12,-0.84)(2.5,-2.88)
\psline[linewidth=1.6pt](2.5,-2.88)(1.82,2.8)
\psline[linewidth=1.6pt](1.82,2.8)(4.08,-0.06)
\psline[linewidth=1.6pt](4.08,-0.06)(2.5,-2.88)
\psline[linewidth=1.6pt](2.5,-2.88)(1.82,2.8)
\rput[tl](1.42,0.8){$\red\large{ \mathcal{G} }$}
\rput(1.8,3.1){$\boldsymbol D$}
\rput(2.5,-3.2){$\boldsymbol A$}
\rput(-1.5,-0.8){$\boldsymbol B$}
\rput(4.4,0){$\boldsymbol C$}
\end{pspicture*}
}
\caption{}\label{Fig. 3}
\end{figure}

When $\mathcal{G}={\rm id}_E$, we have $\varepsilon_{i}\mathcal{S}_i=\varepsilon_{i-1}\mathcal{T}_i$ and there is an identification between upper edges of the tetrahedron $ABCD$ and their projections onto the face $ABC$ (see Figure \ref{Fig. 3}). In other words, in this case the tetrahedron degenerates into a (flat) triangle.

\

The next proposition shows the behaviour of $\mathcal{G}$ and $\mathcal{T}_i$, $i=1,2,3$, under the bilinear form $\langle .,. \rangle$.

\begin{prop}\label{epsilon_hermiticity}
    Let $(\mathcal{S}_1, \mathcal{S}_2, \mathcal{S}_3)$ be an $\boldsymbol{\varepsilon}$-hypersymplectic structure on a pre-Courant algebroid $(E, \langle .,. \rangle, \Theta)$. The maps $\mathcal{G}$ and $\mathcal{T}_i$, $i=1,2,3$, satisfy
$$\langle \mathcal{G}\mathcal{T}_i (X),\mathcal{T}_i (Y)\rangle=\varepsilon_{i-1}\varepsilon_{i+1}\langle \mathcal{G}(X), Y\rangle,$$
for all $X,Y \in \Gamma(E)$.
\end{prop}
\begin{proof}
Using Proposition \ref{propriedades_Ti} i) and ii) and Proposition \ref{propriedades_Ti_Si_G} iii) we have:
$$\langle \mathcal{G}\mathcal{T}_i (X),\mathcal{T}_i (Y)\rangle=\varepsilon_1\varepsilon_2\varepsilon_3\langle \mathcal{T}_i\mathcal{G}\mathcal{T}_i (X),Y\rangle=\varepsilon_1\varepsilon_2\varepsilon_3\langle \mathcal{G}\mathcal{T}_i^2 (X),Y\rangle=\varepsilon_{i-1}\varepsilon_{i+1}\langle \mathcal{G}(X), Y\rangle.$$

\end{proof}

\section{Hypersymplectic structures on deformed pre-Courant algebroids}
\label{section_4}

The results of Section \ref{section_2} (Propositions \ref{propriedades_Ti}, \ref{propriedades_G} and \ref{propriedades_Ti_Si_G}) show that the value of the parameter $\varepsilon_1\varepsilon_2\varepsilon_3=\pm 1$ is determinant for some basic properties of the morphisms $\mathcal{T}_i$, $\mathcal{S}_j$ and $\mathcal{G}$, $i,j=1,2,3$, and for the relations between them. In this section we consider an $\boldsymbol{\varepsilon}$-hypersymplectic structure $(\mathcal{S}_1, \mathcal{S}_2, \mathcal{S}_3)$  on a pre-Courant algebroid $(E, \Theta)$ such that $\varepsilon_1 \varepsilon_2 \varepsilon_3=-1$. We prove that the  $\mathcal{T}_i$'s are Nijenhuis morphisms and we show that we may deform the pre-Courant structure $\Theta$ by $\mathcal{S}_i$ or by $\mathcal{T}_i$, without loosing the property of $(\mathcal{S}_1, \mathcal{S}_2, \mathcal{S}_3)$ being hypersymplectic.

\begin{defn}
    Let $(\mathcal{S}_1, \mathcal{S}_2, \mathcal{S}_3)$ be an {\large$\boldsymbol{\varepsilon}$}-hypersymplectic structure on a pre-Courant algebroid $(E, \Theta)$, such that $\varepsilon_1\varepsilon_2\varepsilon_3=-1$.
    \begin{itemize}
        \item If $\varepsilon_1=\varepsilon_2=\varepsilon_3=-1$, then $(\mathcal{S}_1, \mathcal{S}_2, \mathcal{S}_3)$ is said to be a \emph{hypersymplectic} structure on $(E, \Theta)$.
        \item Otherwise, $(\mathcal{S}_1, \mathcal{S}_2, \mathcal{S}_3)$ is said to be a \emph{para-hypersymplectic} structure on $(E, \Theta)$.
    \end{itemize}
\end{defn}

Note that all para-hypersymplectic structures satisfy, eventually after a cyclic permutation of the indices, $\varepsilon_1=\varepsilon_2=1$ and $\varepsilon_3=-1$. In the sequel, every para-hypersymplectic structure will be considered in such form.

As a direct consequence of Propositions \ref{Si_Nijenhuis}, \ref{torsion_composition} and \ref{propriedades_Ti} ii), we get the following:

\begin{thm}  \label{T_i_Nijenhuis}
    Let $(\mathcal{S}_1, \mathcal{S}_2, \mathcal{S}_3)$ be a (para-)hypersymplectic structure on a pre-Courant algebroid $(E, \Theta)$.
    Then, for each  $i=1,2,3$, the transition morphism $\mathcal{T}_i$ is a Nijenhuis morphism. Moreover,
    \begin{enumerate}
    \item if $\varepsilon_i=-1$, $\mathcal{T}_i$ is a complex structure;
    \item if $\varepsilon_i=1$, $\mathcal{T}_i$ is a para-complex structure.
    \end{enumerate}
\end{thm}

In the case where $\varepsilon_1=\varepsilon_2=\varepsilon_3=-1$ and $(E, \Theta)$ is a Courant algebroid, the triplet $(\mathcal{T}_1, \mathcal{T}_2, \mathcal{T}_3)$ is a \emph{hypercomplex structure} on $(E, \Theta)$ in the sense of \cite{Stienon}, see Remark~\ref{stienon}.

\

In \cite{ALC11} the notion of \emph{Nijenhuis pair} on a pre-Courant algebroid $(E,\Theta)$ was introduced as a pair $(\mathcal{I},\mathcal{J})$  of anti-commuting Nijenhuis morphisms such that $C_\Theta(\mathcal{I},\mathcal{J})=0$.

     Proposition \ref{C_theta(I,J)=0} ii) shows that when two Nijenhuis morphisms $\mathcal{I}$ and $\mathcal{J}$ are (para-)complex structures, i.e., $\lambda_{\mathcal{I}}=\pm 1$ and $\lambda_{\mathcal{J}}=\pm 1$, it is enough that they anti-commute to form a Nijenhuis pair. This is the case when we have a pre-Courant algebroid equipped with a (para-)hypersymplectic structure, as it is stated in the next proposition.

\begin{prop} \label{prop_Nijenhuis_pairs}
   Let $(\mathcal{S}_1, \mathcal{S}_2, \mathcal{S}_3)$ be a (para-)hypersymplectic structure on a pre-Courant algebroid $(E, \Theta)$. Then, $(\mathcal{S}_i,\mathcal{S}_j), (\mathcal{T}_i,\mathcal{T}_j)$ and $(\mathcal{S}_i,\mathcal{T}_j)$ are Nijenhuis pairs, for all $i, j\in\{1,2,3\}, i\neq j$.
\end{prop}

Next we show that when a triplet $(\mathcal{S}_1, \mathcal{S}_2, \mathcal{S}_3)$ is a (para-)hypersymplectic structure on a pre-Courant algebroid $(E, \Theta)$, it is also a (para-)hypersymplectic structure on the pre-Courant algebroid deformed by $\mathcal{T}_i$ or by $\mathcal{S}_i$.

\begin{thm} \label{thm_deformation_Theta}
   Let $(E, \Theta)$ be a pre-Courant algebroid. The following assertions are equivalent:
   \begin{enumerate}
    \item $(\mathcal{S}_1, \mathcal{S}_2, \mathcal{S}_3)$ is a (para-)hypersymplectic structure on $(E, \Theta)$;
    \item $(\mathcal{S}_1, \mathcal{S}_2, \mathcal{S}_3)$ is a (para-)hypersymplectic structure on $(E, \Theta_{\mathcal{S}_i})$;
    \item $(\mathcal{S}_1, \mathcal{S}_2, \mathcal{S}_3)$ is a (para-)hypersymplectic structure on $(E, \Theta_{\mathcal{T}_j})$,
   \end{enumerate}
   $i, j \in \{1,2,3\}$, where $\mathcal{T}_j$ is defined by (\ref{def_transition}).
\end{thm}
\begin{proof}
    Let us prove $(i)\Leftrightarrow(ii)$. By Remark~\ref{rem_equiv_cond_def_epsilonHS_pre-Courant}, it is enough to show that \mbox{${\text{\Fontlukas T}}_{\Theta}\mathcal{S}_j=0 \Leftrightarrow {\text{\Fontlukas T}}_{\Theta_{\mathcal{S}_i}}\mathcal{S}_j=0$}, $i,j\in\{1,2,3\}$. We consider two cases:
    \begin{enumerate}
        \item[a)] The case $i=j$ is a direct consequence of Proposition~\ref{prop_T_theta(N)=0<->T_theta_N(N)=0};
        \item[b)] For the case $i \neq j$, we use Propositions~\ref{prop_Nijenhuis_pairs} and \ref{prop_T_theta(J)=0<->T_theta_I(J)=0}.
    \end{enumerate}
To prove the equivalence $(i)\Leftrightarrow(iii)$ it is enough to notice that $\mathcal{T}_j=\varepsilon_{j-1}\mathcal{S}_{j-1}\mathcal{S}_{j+1}$ and to use twice the equivalence $(i)\Leftrightarrow(ii)$.

\end{proof}

It is known that the deformation $\Theta_I$ of a Courant structure $\Theta$ by a Nijenhuis morphism $\mathcal{I}$ is a Courant structure \cite{grab}. So, if the pre-Courant algebroid $(E, \Theta)$ in Theorem \ref{thm_deformation_Theta} is in particular a Courant algebroid, then $(E, \Theta_{\mathcal{S}_i})$ and $(E, \Theta_{\mathcal{T}_j})$, $i, j \in \{1,2,3\}$, are also Courant algebroids.

\section{One-to-one correspondence}
\label{section_5}

In this section we keep considering an $\boldsymbol{\varepsilon}$-hypersymplectic structure $(\mathcal{S}_1, \mathcal{S}_2, \mathcal{S}_3)$  on a pre-Courant algebroid such that $\varepsilon_1 \varepsilon_2 \varepsilon_3=-1$. We define hyperk\"{a}hler structures on pre-Courant algebroids and prove a one-to-one correspondence between hypersymplectic and hyperk\"{a}hler structures.  We show how we may switch the roles of the morphisms $\mathcal{S}_i$ and $\mathcal{T}_i$ to get a set of equivalent structures; these structures are summarized in a diagram at the end of the section.

Given a pre-Courant algebroid, we may define, in a natural way, a notion of (pseudo-)metric.

\begin{defn}\label{def_pseudo-metric}
    A \emph{pseudo-metric} on a pre-Courant algebroid $(E, \langle .,. \rangle, \Theta)$ is a symmetric and orthogonal bundle automorphism $G:E \to E$. If moreover $G$ is \emph{positive definite}, that is, $\langle G(e),e\rangle > 0$, for all non vanishing sections $e \in \Gamma(E)$, then the prefix ``pseudo'' is removed and $G$ is said to be a \emph{metric} on $(E, \langle .,. \rangle, \Theta)$.
\end{defn}


\begin{rem}\label{G^2=id_E}
In Definition \ref{def_pseudo-metric}, because $G$ is symmetric, the orthogonality condition ($GG^*={\rm id}_E$) can be replaced by an almost para-complex condition ($G^2={\rm id}_E$).
\end{rem}

The next proposition follows directly from the Proposition~\ref{propriedades_G} and Remark \ref{G^2=id_E}.

\begin{prop}
        Let $(\mathcal{S}_1, \mathcal{S}_2, \mathcal{S}_3)$ be a (para-)hypersymplectic structure on a pre-Courant algebroid $(E, \langle .,. \rangle, \Theta)$. Then, the morphism $\mathcal{G}$ given by equation (\ref{def_mathcalG}) is a pseudo-metric on $E$.
\end{prop}

Next, we define the notions of hermitian and para-hermitian pair on a pre-Courant algebroid.

\begin{defn}\label{def_ParaHermitian_preCourant}
    A \emph{hermitian} (resp., \emph{para-hermitian}) \emph{pair}\footnote{Rigourously, we should say \emph{pseudo-hermitian} and \emph{para-pseudo-hermitian} but, in order to simplify the terminology, we omit the prefix ``pseudo''.}  on a pre-Courant algebroid $(E, \langle .,. \rangle, \Theta)$ is a pair $(\mathcal{J},G)$ where $\mathcal{J}$ is a complex (resp., para-complex) structure and $G$ is a pseudo-metric such that, for all $X, Y \in \Gamma(E)$,
    \begin{equation}\label{para-hermiticity_condition}
    \langle G(\mathcal{J}X),\mathcal{J}Y\rangle =\langle G(X), Y\rangle, \quad\quad\quad \left(\text{resp., } \langle G(\mathcal{J}X),\mathcal{J}Y\rangle =-\langle G(X), Y\rangle\right).
    \end{equation}
\end{defn}

\begin{rem}\label{GJ=JG}
    In Definition \ref{def_ParaHermitian_preCourant}, if $\mathcal{J}$ is a skew-symmetric complex (resp. para-complex) structure then
    condition (\ref{para-hermiticity_condition}) is equivalent to $G\mathcal{J}=\mathcal{J}G$.
\end{rem}


As a direct consequence of Proposition \ref{epsilon_hermiticity}, we have the following:

\begin{prop}
 Let $(\mathcal{S}_1, \mathcal{S}_2, \mathcal{S}_3)$ be a (para-)hypersymplectic structure on a pre-Courant algebroid $(E, \langle .,. \rangle, \Theta)$.
        \begin{enumerate}
        \item If $(\mathcal{S}_1, \mathcal{S}_2, \mathcal{S}_3)$ is a hypersymplectic structure, then $(\mathcal{T}_i, \mathcal{G})$ is a hermitian pair, for all $i=1,2,3$.
        \item If $(\mathcal{S}_1, \mathcal{S}_2, \mathcal{S}_3)$ is a para-hypersymplectic structure, then $(\mathcal{T}_1, \mathcal{G})$ and $(\mathcal{T}_2, \mathcal{G})$ are para-hermitian pairs while $(\mathcal{T}_3, \mathcal{G})$ is a hermitian pair.
         \end{enumerate}
\end{prop}

Let us define (para-)hyperk\"{a}hler structures on a pre-Courant algebroid\footnote{In \cite{BCG08}, hyperk\"{a}hler structures on Courant algebroids are called generalized hyper-K\"{a}hler structures.} and see how they are related to (para-)hypersymplectic structures.

\begin{defn}\label{def_PHK_preCourant}
    A quadruple $(\mathcal{T}_1, \mathcal{T}_2, \mathcal{T}_3, \mathcal{G})$ of endomorphisms on a pre-Courant algebroid $(E, \langle .,. \rangle, \Theta)$ is a \emph{hyperk\"{a}hler} (resp., \emph{para-hyperk\"{a}hler}) if:
                \begin{enumerate}
                    \item $\mathcal{G}$ is a pseudo-metric;
                    \item $\mathcal{T}_1$ and $\mathcal{T}_2$ are anti-commuting complex (resp., para-complex) endomorphisms and $\mathcal{T}_3 = \mathcal{T}_1 \mathcal{T}_2$;
                    \item $(\mathcal{G}, \mathcal{T}_j)_{j=1,2}$ are hermitian (resp., para-hermitian) pairs;
                    \item ${\text{\Fontlukas T}}_{\Theta}(\mathcal{G}\mathcal{T}_j)=0,\ j=1,2,3$.
               \end{enumerate}
\end{defn}

Notice that, when $(\mathcal{T}_1, \mathcal{T}_2, \mathcal{T}_3, \mathcal{G})$ is a (para-)hyperk\"{a}hler structure,
$(\mathcal{G}, \mathcal{T}_3)$ is a hermitian pair and the morphisms $\mathcal{T}_1, \mathcal{T}_2$ and $\mathcal{T}_3$ pairwise anti-commute.

\begin{thm}\label{thm_1_1_corresp}
    The triplet $(\mathcal{S}_1, \mathcal{S}_2, \mathcal{S}_3)$ is a hypersymplectic (resp., para-hypersymplectic) structure on a pre-Courant algebroid $(E, \Theta)$ if and only if $(\mathcal{T}_1, \mathcal{T}_2, \mathcal{T}_3, \mathcal{G})$ is a hyperk\"{a}hler (resp., para-hyperk\"{a}hler) structure on $(E, \Theta)$, with $\mathcal{S}_i=\varepsilon_i\varepsilon_{i-1}\mathcal{G}\mathcal{T}_i$, $i=1,2,3$.
\end{thm}
\begin{proof}
    If $(\mathcal{S}_1, \mathcal{S}_2, \mathcal{S}_3)$ is a (para-)hypersymplectic structure on $(E, \Theta)$ then, from what we have proved so far, we easily conclude that $(\mathcal{T}_1, \mathcal{T}_2, \mathcal{T}_3, \mathcal{G})$ is a (para-)hyperk\"{a}hler structure on $(E, \Theta)$.

    Let us now assume that $(\mathcal{T}_1, \mathcal{T}_2, \mathcal{T}_3, \mathcal{G})$ is a (para-)hyperk\"{a}hler structure on $(E, \Theta)$.
    Then, we have
    $$\mathcal{S}_i^2=\mathcal{G}\mathcal{T}_i\mathcal{G}\mathcal{T}_i=\mathcal{G}^2\mathcal{T}_i^2=\mathcal{T}_i^2=\varepsilon_i {\rm id}_E,$$
    where we used the fact that $\mathcal{G}$ and $\mathcal{T}_i$ commute and $\mathcal{G}^2={\rm id}_E$ (see Remarks \ref{GJ=JG} and \ref{G^2=id_E}).
    Moreover,
    $$\mathcal{S}_i\mathcal{S}_{i+1}=\varepsilon_i\varepsilon_{i-1}\varepsilon_{i+1}\varepsilon_i\mathcal{G}\mathcal{T}_i\mathcal{G}\mathcal{T}_{i+1}
        =\varepsilon_{i-1}\varepsilon_{i+1}\mathcal{T}_i\mathcal{T}_{i+1}=-\varepsilon_{i-1}\varepsilon_{i+1}\mathcal{T}_{i+1}\mathcal{T}_i
        =-\mathcal{S}_{i+1}\mathcal{S}_i.$$
    Finally, because $\mathcal{S}_i^2=\varepsilon_i {\rm id}_E$ and ${\text{\Fontlukas T}}_{\Theta}\mathcal{S}_i=0$ (see item iv) of Definition \ref{def_PHK_preCourant}) we conclude that $\Theta_{\mathcal{S}_i,\mathcal{S}_i}=\varepsilon_i \Theta$. Therefore, $(\mathcal{S}_1, \mathcal{S}_2, \mathcal{S}_3)$ is a (para-)hypersymplectic structure on $(E, \Theta)$.
\end{proof}

Next, we see that the tetrahedron model (see Figure \ref{Fig. 3}), besides being an efficient way to summarize all the algebraic relations between the morphisms  of a (para-)hypersymplectic structure, is an accurate representation that enables us to discover new relations. In fact, the next theorem shows that the symmetries of the tetrahedron are symmetries of the (para-)hypersymplectic structures on pre-Courant algebroids. These symmetries do not exist for (para-)hypersymplectic structures on Lie algebroids (see definition in \cite{AC13}).

\begin{thm}\label{thm_invariant_by_rotation}
    The triplet $(\mathcal{S}_1, \mathcal{S}_2, \mathcal{S}_3)$ is a hypersymplectic (resp., para-hypersymplectic) structure on a pre-Courant algebroid $(E, \Theta)$ if and only if $(\mathcal{S}_1, \mathcal{T}_2, \mathcal{T}_3)$ is a hypersymplectic (resp., para-hypersymplectic) structure on $(E, \Theta)$. Furthermore, both (para-)hypersymplectic structures determine equal or opposite pseudo-metrics.
\end{thm}
\begin{proof}
    If $(\mathcal{S}_1, \mathcal{S}_2, \mathcal{S}_3)$ is a (para-)hypersymplectic structure then, previous definitions and results yield,
    $$\left\{
      \begin{array}{l}
        \mathcal{S}_i^2=\mathcal{T}_i^2=\varepsilon_i {\rm id}_E;\\
        \mathcal{S}_{i}\mathcal{S}_j+\mathcal{S}_{j}\mathcal{S}_i=\mathcal{T}_{i}\mathcal{T}_j+\mathcal{T}_{j}\mathcal{T}_i
            =\mathcal{S}_{i}\mathcal{T}_j+\mathcal{T}_{j}\mathcal{S}_i=0;\\
        {\text{\Fontlukas T}}_{\Theta}\mathcal{S}_i={\text{\Fontlukas T}}_{\Theta}\mathcal{T}_i=0.
      \end{array}
    \right.$$
Thus, $(\mathcal{S}_1, \mathcal{T}_2, \mathcal{T}_3)$ is a (para-)hypersymplectic structure.

Now, let us assume that $(\mathcal{S}_1, \mathcal{T}_2, \mathcal{T}_3)$ is a (para-)hypersymplectic structure. In this case, the transition morphisms are $\varepsilon_{1}\varepsilon_{3}\mathcal{T}_1, \varepsilon_{1}\varepsilon_{3}\mathcal{S}_2$ and $\varepsilon_{1}\varepsilon_{3}\mathcal{S}_3$ and, using the first part of the proof, we conclude that $(\mathcal{S}_1, \varepsilon_{1}\varepsilon_{3}\mathcal{S}_2, \varepsilon_{1}\varepsilon_{3}\mathcal{S}_3)$ is a (para-)hypersymplectic structure. Therefore, $(\mathcal{S}_1, \mathcal{S}_2,\mathcal{S}_3)$ is a (para-)hypersymplectic structure.

Finally, because $\mathcal{S}_3\mathcal{S}_2\mathcal{S}_1=\varepsilon_{1}\varepsilon_{3}\mathcal{T}_3\mathcal{T}_2\mathcal{S}_1$, the pseudo-metrics induced by both (para-)hypersymplectic structures are equal or opposite.
\end{proof}

Applying successively Theorems \ref{thm_1_1_corresp} and \ref{thm_invariant_by_rotation}, we conclude that one (para-)hypersymplectic structure $(\mathcal{S}_1, \mathcal{S}_2,\mathcal{S}_3)$ on $(E, \Theta)$ induces several (para-)hypersymplectic and (para-)hyperk\"{a}hler structures on $(E, \Theta)$, as the next diagram shows.

$$\xymatrixcolsep{4pc}\xymatrix{
    {(\mathcal{S}_1, \mathcal{S}_2, \mathcal{S}_3)\textrm{ (para-)hypersymplectic}} \ar@{<->}[r]^{\scriptsize\ \ Thm\ \ref{thm_1_1_corresp}} \ar@{<->}[d]_{\scriptsize Thm\ \ref{thm_invariant_by_rotation}}&  {(\mathcal{T}_1, \mathcal{T}_2, \mathcal{T}_3, \mathcal{G})\textrm{ (para-)hyperk\"{a}hler}} \ar@{<->}[d]\\
    {(\mathcal{S}_1, \mathcal{T}_2, \mathcal{T}_3)\textrm{ (para-)hypersymplectic}} \ar@{<->}[r]^{\scriptsize\ \ Thm\ \ref{thm_1_1_corresp}} \ar@{<->}[d]_{\scriptsize Thm\ \ref{thm_invariant_by_rotation}}&  {(\mathcal{T}_1, \mathcal{S}_2, \mathcal{S}_3, \mathcal{G})\textrm{ (para-)hyperk\"{a}hler}} \ar@{<->}[d]\\
    {(\mathcal{T}_1, \mathcal{S}_2, \mathcal{T}_3)\textrm{ (para-)hypersymplectic}} \ar@{<->}[r]^{\scriptsize\ \ Thm\ \ref{thm_1_1_corresp}} \ar@{<->}[d]_{\scriptsize Thm\ \ref{thm_invariant_by_rotation}}&  {(\mathcal{S}_1, \mathcal{T}_2, \mathcal{S}_3, \mathcal{G})\textrm{ (para-)hyperk\"{a}hler}} \ar@{<->}[d]\\
    {(\mathcal{T}_1, \mathcal{T}_2, \mathcal{S}_3)\textrm{ (para-)hypersymplectic}} \ar@{<->}[r]^{\scriptsize\ \ Thm\ \ref{thm_1_1_corresp}}&  {(\mathcal{S}_1, \mathcal{S}_2, \mathcal{T}_3, \mathcal{G})\textrm{ (para-)hyperk\"{a}hler}}
    }$$

\


\section{Hypersymplectic structures on Lie algebroids}
\label{section_6}

The purpose of this section is to present a first example of an $\boldsymbol{\varepsilon}$-hypersymplectic structure on a Courant algebroid, which is constructed out of an $\boldsymbol{\varepsilon}$-hypersymplectic structure on a Lie algebroid. First, we recall the definition and some properties of the latter \cite{A10, AC13}.


 An $\boldsymbol{\varepsilon}${\em -hypersymplectic structure} on a Lie algebroid $(A, \mu)$ is a triplet $(\omega_1, \omega_2, \omega_3)$ of symplectic forms with inverse Poisson bivectors $(\pi_1, \pi_2, \pi_3)$ such that the transition endomorphisms $N_1$, $N_2$ and $N_3$ on A, defined by
 \begin{equation} \label{transition_A}
 N_i:=\pi^\#_{i-1} \circ \omega^\flat_{i+1},\quad i \in \mathbb Z_3,\end{equation}
 satisfy
\begin{equation} \label{transitions_A}
N_i^2 = \varepsilon_i {\mathrm{id}}_A, \quad i=1,2,3.
\end{equation}
An important property of the transitions morphisms $N_i$, $i=1,2,3$, is that they are Nijenhuis morphisms.

Having an $\boldsymbol{\varepsilon}$-hypersymplectic structure on a Lie algebroid $(A, \mu)$, we define  $g \in \bigotimes^2 A^*$ by setting, for all $X,Y \in \Gamma(A)$, $$g(X,Y):=\langle g^\flat X, Y\rangle,$$ where $g^\flat: A\longrightarrow A^*$ is given by
    \begin{equation}\label{first_defn _g}
    g^\flat:=\varepsilon_3\varepsilon_2\ {\omega_3}^{\flat} \circ {\pi_1}^{\sharp} \circ {\omega_2}^{\flat}.
    \end{equation}
The definition of $g^\flat$ is not affected by a circular permutation of the indices in equation~(\ref{first_defn _g}), that is,
\begin{equation} \label{second_defn_g}
g^\flat=\varepsilon_{i-1}\varepsilon_{i+1}\ {\omega_{i-1}}^{\flat} \circ {\pi_i}^{\sharp} \circ {\omega_{i+1}}^{\flat},
\end{equation}
for all $i \in \mathbb{Z}_3$.
Moreover, we have
\begin{equation*}  \label{skew_or_not}
(g^\flat)^* = - \varepsilon_1\varepsilon_2\varepsilon_3\ g^\flat,
\end{equation*}
which means that
 $g$ is symmetric or skew-symmetric, depending on the sign of the product $\varepsilon_1\varepsilon_2\varepsilon_3$.
When $\varepsilon_1 \varepsilon_2 \varepsilon_3=-1$, the morphism $g^\flat$ defined by (\ref{first_defn _g}) determines a {\em pseudo-metric} on $A$.

\

Let $(A,\mu)$ be a Lie algebroid and consider the Courant algebroid $(A\oplus A^*,\mu)$. If we take a triplet $(\omega_1, \omega_2, \omega_3)$ of $2$-forms and a triplet $(\pi_1,\pi_2, \pi_3)$ of bivectors on $A$, we may define the skew-symmetric bundle endomorphisms \mbox{$\mathcal{S}_i: A \oplus A^* \to A \oplus A^*$}, $i=1,2,3$,
\begin{equation} \label{definition_Si}
\mathcal{S}_i:=\left[
                    \begin{array}{ccc}
                    0&\ &\varepsilon_i\,\pi_{i}^{\sharp}\\
                    \omega_{i}^{\flat}&\ &0
                    \end{array}
\right].
\end{equation}
In order to simplify the writing and if there is no risk of confusion, we shall omit the symbols $\sharp$ and $\flat$ and denote the morphisms $\omega_{i}^{\flat}$ and $\pi_{i}^{\sharp}$ by $\omega_{i}$ and $\pi_{i}$, respectively. Moreover, in the supergeometric setting, we have
\begin{equation*}
\mathcal{S}_i(X+ \alpha)= \{X+ \alpha, \omega_i + \varepsilon_i \pi_i \},
\end{equation*}
for all $X + \alpha \in A \oplus A^*$.

\begin{lem} \label{lem_algebraic_conditions}
Let $\omega_1$, $\omega_2$ and $\omega_3$  be $2$-forms on a vector bundle $A$ over $M$ and $\pi_1$, $\pi_2$ and $\pi_3$ bivectors on $A$. Consider the vector bundle morphisms $N_1$, $N_2$ and $N_3$ on $A$,  given by (\ref{transition_A}), and the bundle endomorphisms $\mathcal{S}_1$, $\mathcal{S}_2$ and $\mathcal{S}_3$ on $A \oplus A^*$, given by (\ref{definition_Si}). Then, for all $i=1,2,3$,
\begin{enumerate}
\item
$\mathcal{S}_i^2= \varepsilon_i \rm{id}_{A \oplus A^*} \, \Leftrightarrow \, \pi_i \circ \omega_i=\rm{id}_A$,
\item  $\mathcal{S}_{i-1}\mathcal{S}_{i+1}=\varepsilon_1\varepsilon_2\varepsilon_3 \, \mathcal{S}_{i+1}\mathcal{S}_{i-1} \, \Leftrightarrow \, N_i^2= \varepsilon_i \rm{id}_A.$
\end{enumerate}
\end{lem}

\begin{proof}
A simple computation gives i). To prove ii), we notice that $N_i^2= \varepsilon_i \, \rm{id}_A$ is equivalent to
$$\omega_{i+1} \circ \pi_{i-1}= \varepsilon_i \, \omega_{i-1} \circ \pi_{i+1},$$ 
for all $i\in \mathbb Z_3$ (see \cite{AC13}). On the other hand, we have
\begin{equation*}
\mathcal{S}_{i-1}\mathcal{S}_{i+1} = \left[
                    \begin{array}{ccc}
                    \varepsilon_{i-1} \, \pi_{i-1} \circ \omega_{i+1}&\ & 0\\
                    0&\ &\varepsilon_{i+1} \, \omega_{i-1}\circ \pi_{i+1}
                    \end{array}
\right]
\end{equation*}
and
\begin{multline*}
\varepsilon_1 \varepsilon_2 \varepsilon_3 \mathcal{S}_{i+1}\mathcal{S}_{i-1} = \varepsilon_1 \varepsilon_2 \varepsilon_3 \left[
                    \begin{array}{ccc}
                    \varepsilon_{i+1} \, \pi_{i+1} \circ \omega_{i-1}&\ & 0\\
                    0&\ &\varepsilon_{i-1} \, \omega_{i+1}\circ \pi_{i-1}
                    \end{array}
\right]\\
=\left[
                    \begin{array}{ccc}
                    \varepsilon_{i-1} \varepsilon_i \, \pi_{i+1} \circ \omega_{i-1}&\ & 0\\
                    0&\ &\varepsilon_{i+1} \varepsilon_i \, \omega_{i+1}\circ \pi_{i-1}
                    \end{array}
\right].
\end{multline*}
So, $\mathcal{S}_{i-1}\mathcal{S}_{i+1}=\varepsilon_1 \varepsilon_2 \varepsilon_3 \mathcal{S}_{i+1}\mathcal{S}_{i-1}$ if and only if
$\omega_{i+1} \circ \pi_{i-1}= \varepsilon_i \, \omega_{i-1} \circ \pi_{i+1}$ and this completes the proof.
\end{proof}

\begin{thm}  \label{structure_A+A*_mu}
A triplet $(\omega_1, \omega_2, \omega_3)$, with inverse  $(\pi_1, \pi_2, \pi_3)$, is an $\boldsymbol{\varepsilon}$-hypersymplectic structure on a Lie algebroid $(A,\mu)$ if and only if  the triplet $(\mathcal{S}_1, \mathcal{S}_2, \mathcal{S}_3)$ is an $\boldsymbol{\varepsilon}$-hypersymplectic structure on the Courant algebroid $(A\oplus A^*,\mu)$, with $\mathcal{S}_i$, $i=1,2,3$, given by (\ref{definition_Si}).
\end{thm}
\begin{proof}
Suppose that $(\omega_1, \omega_2, \omega_3)$ is an $\boldsymbol{\varepsilon}$-hypersymplectic structure on a Lie algebroid $(A,\mu)$ and $\pi_i$ is the inverse of $\omega_i$, $i=1, 2, 3$.  According to Definition~\ref{def_epsilonHS_pre-Courant} and Lemma~\ref{lem_algebraic_conditions}, we only have to check that $\mu_{\mathcal{S}_i,\mathcal{S}_i}=\varepsilon_i \mu$, for $i=1,2,3$.
Recalling that $\pi$ is a Poisson bivector if and only if $\{ \pi,\{\pi, \mu \} \}=0$ and $\omega$ is a closed $2$-form if and only if $\{ \mu, \omega \}=0$, a simple computation gives:
\begin{eqnarray*}
\{\mathcal{S}_i, \{ \mathcal{S}_i, \mu \} \}&=& \{\omega_i + \varepsilon_i \pi_i, \{\omega_i + \varepsilon_i \pi_i, \mu \} \}= \varepsilon_i \{\omega_i, \{\pi_i, \mu \} \}\\
&=& \varepsilon_i \, \mu,
\end{eqnarray*}
where we used, in the last equality, the formula
\begin{equation} \label{identity}
\{{\rm{id}}_A, \chi\}=(q-p)\chi, \quad \chi \in {\mathcal F}_{A \oplus A^*}^{(p,q)}.
\end{equation}

Conversely, assume that the endomorphisms $\mathcal{S}_i=\left[
                    \begin{array}{ccc}
                    0&\ &\varepsilon_i\,\pi_{i}\\
                    \omega_{i}&\ &0
                    \end{array}
\right]$, $i=1,2,3$, form an $\boldsymbol{\varepsilon}$-hypersymplectic structure on the Courant algebroid $(A\oplus A^*,\mu)$. Using again Lemma~\ref{lem_algebraic_conditions}, we only have to prove that the non-degenerate $2$-forms $\omega_i$ are symplectic.  From
$$\{\omega_i + \varepsilon_i \pi_i, \{\omega_i + \varepsilon_i \pi_i, \mu \} \}=\varepsilon_i \mu,$$
we get $\{ \pi_i,\{\pi_i, \mu \} \}=0$, which means that $\pi_i$ is a Poisson bivector on $(A, \mu)$. But $\pi_i$ being a Poisson bivector on $(A, \mu)$ is equivalent to $\omega_i$ being a symplectic form on $(A, \mu)$.
\end{proof}

Under the conditions of Theorem~\ref{structure_A+A*_mu}, the transition morphisms of the  $\boldsymbol{\varepsilon}$-hypersymplectic structure $(\mathcal{S}_1, \mathcal{S}_2, \mathcal{S}_3)$ on $(A\oplus A^*,\mu)$, defined by  (\ref{def_transition}), are given by
\begin{equation*} \label{def_T_i}
\mathcal{T}_i=\left[
                    \begin{array}{ccc}
                        N_i&\ &0\\
                        0&\ &\varepsilon_1\varepsilon_2\varepsilon_3\, {N_i}^*
                    \end{array}
                \right], \quad i=1,2,3,
\end{equation*}
where $N_i$ is the transition morphism of the $\boldsymbol{\varepsilon}$-hypersymplectic structure $(\omega_1, \omega_2, \omega_3)$ on the Lie algebroid $(A,\mu)$, see (\ref{transition_A}). The endomorphism $\mathcal{G}:A\oplus A^* \to A\oplus A^*$  defined by
 (\ref{def_mathcalG})
is given by $$\mathcal{G}=\left[
                    \begin{array}{ccc}
                        0&\ &(g^\flat)^{-1}\\
                        g^\flat&\ &0
                    \end{array}
                \right],$$
                where $g^\flat: A \to A^*$ is defined by (\ref{first_defn _g}).

\section{Hypersymplectic structures on Lie bialgebroids}
\label{section_7}

In this section we present a class of examples of $\boldsymbol{\varepsilon}$-hypersymplectic structures on a Courant algebroid $(A\oplus A^*,\mu + \gamma)$, which is the double of a Lie bialgebroid $((A,A^*), \mu,\gamma)$.

Having in mind that a bivector $\pi$ on $A$ can be seen as a $2$-form on $A^*$, through the identification $A=(A^*)^*$, we have the following result.

\begin{thm} \label{structure_A+A*_mu+gamma}
Let $((A,A^*), \mu,\gamma)$ be a Lie bialgebroid and $(\mathcal{S}_1, \mathcal{S}_2, \mathcal{S}_3)$ be a triplet of bundle endomorphisms of $A \oplus A^*$, with $\mathcal{S}_i$ given by (\ref{definition_Si}). The following assertions are equivalent:
\begin{enumerate}
  \item $(\mathcal{S}_1, \mathcal{S}_2, \mathcal{S}_3)$ is an $\boldsymbol{\varepsilon}$-hypersymplectic structure on the Courant algebroid $(A\oplus A^*,\mu + \gamma)$
  \item $(\omega_1, \omega_2, \omega_3)$ is an $\boldsymbol{\varepsilon}$-hypersymplectic structure on the Lie algebroid $(A,\mu)$, $(\pi_1, \pi_2, \pi_3)$ is an $\boldsymbol{\varepsilon}$-hypersymplectic structure on the Lie algebroid $(A^*, \gamma)$ and $\pi_i$ is the inverse of $\omega_i$, $i=1, 2, 3$.
\end{enumerate}
\end{thm}
\begin{proof}
We use Lemma~\ref{lem_algebraic_conditions} noticing that $\pi_i \circ \omega_i={\rm{id}}_{A} \Leftrightarrow \omega_i \circ \pi_i={\rm{id}}_{A^*}$ and
  $N_i ^2= \varepsilon_i {\rm{id}}_{A} \Leftrightarrow (N_i^*)^2= \varepsilon_i {\rm{id}}_{A^*}$, $i=1, 2, 3$, so that conditions i) and ii) of Definition~\ref{def_epsilonHS_pre-Courant} are satisfied if and only if $\pi_i$ and $\omega_i$ are inverses of each other and (\ref{transitions_A}) holds. Moreover, using the bidegrees of $\mathcal{F}_{A \oplus A^*}^{3}$, we have
\begin{equation}  \label{mu+gamma}
\{\mathcal{S}_i, \{\mathcal{S}_i, \mu + \gamma \}\}= \varepsilon_i (\mu +\gamma) \Leftrightarrow
\begin{cases}
\{\pi_i, \{\pi_i, \mu \}\}=0\\
\{\omega_i, \{\omega_i, \gamma \}\}=0\\
\{\omega_i, \{\pi_i, \mu \}\}+ \{\pi_i, \{\omega_i, \mu \}\}=\mu\\
\{\omega_i, \{\pi_i, \gamma \}\}+ \{\pi_i, \{\omega_i, \gamma \}\}=\gamma.
\end{cases}
\end{equation}
The first equation on the right-hand side of (\ref{mu+gamma}) means that $\pi_i$ is a Poisson bivector on $(A,\mu)$, which is equivalent to $\omega_i$ being a symplectic form on $(A,\mu)$.
The second equation on the right-hand side of (\ref{mu+gamma}) means that $\omega_i$, seen as a bivector on $A^*$, is  Poisson on the Lie algebroid $(A^*, \gamma)$, which is equivalent to saying that $\pi_i$ is symplectic on $(A^*, \gamma)$. Concerning the third and fourth equations on the right-hand side of (\ref{mu+gamma}),
an easy computation shows that they are equivalent, respectively, to the first and second equations.
\end{proof}

It is well known that a Poisson bivector $\pi$ on $(A,\mu)$ determines a Lie algebroid structure on $A^*$; we denote by $\mu_{\pi}$ this induced structure.
In \cite{AC14} we proved that if  $(\omega_1, \omega_2, \omega_3)$ is an $\boldsymbol{\varepsilon}$-hypersymplectic structure on a Lie algebroid $(A,\mu)$ and
$\pi_i$ is the inverse of $\omega_i$, $i=1, 2, 3$, then the triplet $(\pi_1, \pi_2, \pi_3)$ is an $\boldsymbol{\varepsilon}$-hypersymplectic structure on the Lie algebroid $(A^*, \mu_{\pi_{i}})$.
So, given an $\boldsymbol{\varepsilon}$-hypersymplectic structure $(\omega_1, \omega_2, \omega_3)$ on a Lie algebroid $(A,\mu)$, Theorem~\ref{structure_A+A*_mu+gamma} yields that the triplet $(\mathcal{S}_1, \mathcal{S}_2, \mathcal{S}_3)$ is an $\boldsymbol{\varepsilon}$-hypersymplectic structure on the Courant algebroid $(A\oplus A^*,\mu +\mu_{\pi_{i}})$.
Conversely, if
$(\mathcal{S}_1, \mathcal{S}_2, \mathcal{S}_3)$ is an $\boldsymbol{\varepsilon}$-hypersymplectic structure on the Courant algebroid $(A\oplus A^*,\mu +\mu_{\pi_{i}})$ then, by Theorem~\ref{structure_A+A*_mu+gamma}, $(\omega_1, \omega_2, \omega_3)$ is an $\boldsymbol{\varepsilon}$-hypersymplectic structure on the Lie algebroid $(A,\mu)$.

Thus, we have proved:

\begin{cor}\label{cor_structure_A+A*_mu+mu_pi}
The triplet $(\omega_1, \omega_2, \omega_3)$, with inverse $(\pi_1, \pi_2, \pi_3)$, is an $\boldsymbol{\varepsilon}$-hypersymplectic structure on the Lie algebroid $(A,\mu)$ if and only if $(\mathcal{S}_1, \mathcal{S}_2, \mathcal{S}_3)$ is an $\boldsymbol{\varepsilon}$-hypersymplectic structure on the Courant algebroid $(A\oplus A^*,\mu +\mu_{\pi_{i}})$, with $\mathcal{S}_i$ given by (\ref{definition_Si}),
$i=1, 2, 3$.
\end{cor}


\section{Hypersymplectic structures with torsion on Lie algebroids}
\label{section_8}

In this section we pretend to study a class of examples of hypersymplectic structures on pre-Courant algebroids determined by some structures on Lie algebroids which are called {\em hypersymplectic with torsion}. These are introduced and discussed in \cite{AC14_a} and may be considered as being equivalent to hyperk\"{a}hler structures with torsion, also known as HKT structures \cite{HP96}.
The hypersymplectic structures with torsion on  Lie algebroids provide examples of hypersymplectic structures (without torsion) on Courant algebroids which are doubles of quasi-Lie bialgebroids and even in the more general case where the Courant structure is the double of a proto-Lie bialgebroid.

We give the definition of a hypersymplectic structure with torsion on a Lie algebroid $(A,\mu)$, which is a particular case of an $\boldsymbol{\varepsilon}$-hypersymplectic structure with torsion considered in \cite{AC14_a}.

\

Let $\omega_1, \omega_2$ and $\omega_3$ be nondegenerate $2$-forms on a Lie algebroid $(A,\mu)$,  with inverses  $\pi_1, \pi_2$ and $\pi_3 \in \Gamma(\wedge^2 A)$, respectively, and consider the transition morphisms $N_1,N_2,N_3: A \to A$ given by (\ref{transition_A}).
\begin{defn}\label{def_HST}
    The triplet $(\omega_1, \omega_2, \omega_3)$ is a {\em hypersymplectic structure with torsion} on the Lie algebroid $(A, \mu)$ if
    \begin{equation}\label{eq_Ij_almostCPS}
        {N_i}^2=- {\rm id}_{A}, \quad i=1,2,3, \quad {\rm and} \quad  N_1 {\rm d} \omega_1= N_2 {\rm d} \omega_2= N_3 {\rm d} \omega_3,
    \end{equation}
    where $N_i\d \omega_i(X,Y,Z)= \d \omega_i(N_i X, N_i Y, N_i Z)$, for all $X,Y,Z \in \Gamma(A)$ and $\d$ stands for the differential of the Lie algebroid $(A,\mu)$.
\end{defn}

When the non-degenerate $2$-forms $\omega_1, \omega_2$ and $\omega_3$ are closed, then they are symplectic forms and the right hand side of (\ref{eq_Ij_almostCPS}) is trivially satisfied. In this case, the triplet $(\omega_1, \omega_2, \omega_3)$ is a \emph{hypersymplectic structure} (without torsion) on $(A, \mu)$, that is, an $\boldsymbol{\varepsilon}$-hypersymplectic structure with $\varepsilon_1=\varepsilon_2=\varepsilon_3=-1$ (see Section \ref{section_6}).

The next lemma will be useful in what follows.

\begin{lem}  \label{4equivalent2}
Let $((A, A^*), \mu, \gamma)$ be a Lie bialgebroid, $\psi \in \Gamma(\wedge^3 A)$, $\phi \in \Gamma(\wedge^3 A^*)$, $\pi \in \Gamma(\wedge^2A)$ and $\omega \in \Gamma(\wedge^2 A^*)$, with $\pi$ and $\omega$ inverse of each other. Then,
\begin{enumerate}
\item
$\{\pi, \{\pi, \mu \}\}=2\, \psi \, \Leftrightarrow \, 2 \, \{\pi, \{\omega, \mu \}\}= \{\omega, \{\omega, \psi \}\}$;
\item
$\{\omega, \{\omega, \gamma \}\}=2\, \phi \, \Leftrightarrow \, 2 \, \{\omega, \{\pi, \gamma \}\}= \{\pi, \{\pi, \phi \}\}$.
\end{enumerate}
\end{lem}

\begin{proof}
i) Let us assume that $\{\pi, \{\pi, \mu \}\}=2\, \psi$. Then,
\begin{equation*}
\{\omega,\{\pi, \{\pi, \mu \}\}\}=2\, \{\omega, \psi \}
\end{equation*}
and the Jacobi identity together with (\ref{identity}) gives
\begin{equation*}
\{\pi,\{\pi, \{\omega, \mu \}\}\}=2\, \{\omega, \psi \}.
\end{equation*}
Thus,
\begin{equation*}
\{\omega, \{\pi,\{\pi, \{\omega, \mu \}\}\}\}=2\, \{\omega,\{\omega, \psi \}\}
\end{equation*}
or, equivalently,
\begin{equation}  \label{torsion_quasi_Lie}
\{\pi, \{\omega, \mu \}\} + \{\pi, \{\omega,\{\pi, \{\omega, \mu \}\}\}\}=2\, \{\omega,\{\omega, \psi \}\}.
\end{equation}
Finally, (\ref{torsion_quasi_Lie}) gives
\begin{equation*}
2 \{\pi, \{\omega, \mu \}\} = \{\omega,\{\omega, \psi \}\}.
\end{equation*}

Now, we assume that $\{\omega, \{\omega, \psi \}\}=2\, \{\pi, \{\omega, \mu \}\}$. Then,
\begin{equation*}
\{\pi, \{\omega, \{\omega, \psi \}\} \}=2\, \{\pi, \{\pi, \{\omega, \mu \}\}\}
\end{equation*}
which is equivalent to
\begin{equation*}
\{\omega, \psi \} + \{\omega, \{\pi, \{\omega, \psi \}\} \}=2\, \{\pi, \{\pi, \{\omega, \mu \}\}\}.
\end{equation*}
Thus,
\begin{equation} \label{torsion_quasi_Lie2}
\{\pi, \{\omega, \psi \}\} + \{ \pi, \{\omega, \{\pi, \{\omega, \psi \}\} \}\}=2 \, \{\pi, \{\pi, \{\pi, \{\omega, \mu \}\}\}\}.
\end{equation}
From (\ref{torsion_quasi_Lie2}) we get, applying the Jacobi identity and (\ref{identity}) several times,
\begin{eqnarray*}
&3 \psi + 3 \{ \pi, \{ \omega, \psi \}\}= - 2 \, \{\pi, \{\pi,\mu\}\} + 2\, \{\pi, \{\pi,\{\omega, \{\pi, \mu \}\}\}\}\\
&\Leftrightarrow \,
6 \psi = \{\pi,\{\omega, \{ \pi, \{\pi, \mu \}\}\}\}
\Leftrightarrow \,
2 \psi = \{ \pi, \{\pi, \mu \}\}.
\end{eqnarray*}

\

\noindent ii) The proof is similar to case i).

\end{proof}

Now, we have to mention that the definition of hypersymplectic structure with torsion on a Lie algebroid can be given using the inverses of the non-degenerate $2$-forms $\omega_i$. More precisely,
\begin{center}
\emph{$(\omega_1, \omega_2, \omega_3)$ is a hypersymplectic structure with torsion on $(A,\mu)$ if and only if} \begin{equation}\label{contravariant_def}
N_i^2=- \rm{id}_A \quad \mbox{\it and} \quad [\pi_1, \pi_1]=[\pi_2, \pi_2]=[\pi_3, \pi_3], \end{equation}
\end{center}where $[.,.]$ is the Schouten-Nijenhuis bracket of multivectors on $A$.  The equivalence of the two definitions is proved in \cite{AC14_a}.

The next proposition gives a first example of a hypersymplectic structure on a pre-Courant algebroid, which is constructed out of a hypersymplectic structure with torsion on a Lie algebroid.

\begin{prop}  \label{HS_preCourant_structure_mu+gamma+psi}
Let $(\omega_1, \omega_2, \omega_3)$ be a triplet of $2$-forms and $(\pi_1,\pi_2, \pi_3)$ be a triplet  of bivectors on a Lie algebroid $(A,\mu)$. Consider the triplet $(\mathcal{S}_1, \mathcal{S}_2, \mathcal{S}_3)$ of endomorphisms of $A \oplus A^*$, with $\mathcal{S}_i$ given by (\ref{definition_Si}). The following assertions are equivalent:
\begin{enumerate}
    \item $(\omega_1, \omega_2, \omega_3)$ is a hypersymplectic structure with torsion on the Lie algebroid $(A,\mu)$ and $\pi_i$ is the inverse of $\omega_i$, $i=1,2,3$;
    \item $(\mathcal{S}_1, \mathcal{S}_2, \mathcal{S}_3)$ is a hypersymplectic structure on the pre-Courant algebroid $(A\oplus A^*, \mu+\psi)$, for some $\psi\in \Gamma(\wedge^3 A)$.
\end{enumerate}
\end{prop}

\begin{proof}
Let us assume that $(\omega_1, \omega_2, \omega_3)$ is a hypersymplectic structure with torsion on a Lie algebroid $(A,\mu)$.
From Lemma~\ref{lem_algebraic_conditions}, conditions i) and ii)  of Definition~\ref{def_epsilonHS_pre-Courant} are satisfied while for condition iii), we have
\begin{align}
    \{\mathcal{S}_i, \{\mathcal{S}_i , \mu+\psi \} \}= - \mu -\psi &\Leftrightarrow
    \begin{cases}
        \{ \omega_i, \{\pi_i, \mu \} \}+ \{ \pi_i, \{\omega_i, \mu \} \} - \{ \omega_i, \{\omega_i, \psi \} \} =\mu \\
        -\{\pi_i, \{\pi_i, \mu\}\} + \{ \pi_i, \{\omega_i, \psi \} \}=\psi
    \end{cases}\nonumber\\
    &\Leftrightarrow
    \begin{cases}
        2 \{ \pi_i, \{\omega_i, \mu \} \} = \{ \omega_i, \{\omega_i, \psi \} \} \\
        \{\pi_i, \{\pi_i, \mu\}\}=2 \psi
    \end{cases}\nonumber\\
    &\Leftrightarrow \{\pi_i, \{\pi_i, \mu\}\}=2 \psi,\label{equivalent_conditions}
\end{align}
where the latter equivalence is given by Lemma~\ref{4equivalent2}. Equation (\ref{equivalent_conditions}) exhibits the appropriate definition of $\psi$ in order to satisfy condition iii) of Definition~\ref{def_epsilonHS_pre-Courant}.

Let us now assume that $(\mathcal{S}_1, \mathcal{S}_2, \mathcal{S}_3)$ is a hypersymplectic structure on a pre-Courant algebroid $(A\oplus A^*, \mu+\psi)$, with $\psi\in \Gamma(\wedge^3 A)$. From (\ref{equivalent_conditions}), the $3$-vector $\psi$ is given by $\psi= - \frac{1}{2}[\pi_i, \pi_i]$, $i\in \{1,2,3\}$. Thus, $[\pi_1, \pi_1]=[\pi_2, \pi_2]=[\pi_3, \pi_3]$ and, from Lemma~\ref{lem_algebraic_conditions} and (\ref{contravariant_def}), we get that $(\omega_1, \omega_2, \omega_3)$ is a hypersymplectic structure with torsion on $(A,\mu)$.
\end{proof}

Notice that in the assertion ii) of Proposition~\ref{HS_preCourant_structure_mu+gamma+psi}, since $(\mathcal{S}_1, \mathcal{S}_2, \mathcal{S}_3)$ is a hypersymplectic structure for the pre-Courant structure $\mu + \psi$, condition $(\mu+\psi)_{\mathcal{S}_k, \mathcal{S}_k}=- \mu-\psi$ holds and implies that $\psi$ has to be of the form $-\frac{1}{2}[\pi_k, \pi_k]$, for any $k\in \{1,2,3\}$.

If we aim to have a Courant structure on $A\oplus A^*$, in the statement ii) of Proposition \ref{HS_preCourant_structure_mu+gamma+psi}, we have to require the bivectors to be weak-Poisson\footnote{A bivector $\pi$ on a Lie algebroid $(A,\mu)$ is \emph{weak-Poisson} if $\{ \mu, \{\{\pi, \mu \} , \pi\} \}=0$ or, equivalently, $\{ \mu, [\pi,\pi]\}=0$.} as shown in the next theorem.

\begin{thm}  \label{HS_Courant_structure_mu+gamma+psi}
Let $(\omega_1, \omega_2, \omega_3)$ be a triplet of $2$-forms and $(\pi_1,\pi_2, \pi_3)$ be a triplet  of bivectors on a Lie algebroid $(A,\mu)$. Consider the triplet $(\mathcal{S}_1, \mathcal{S}_2, \mathcal{S}_3)$ of endomorphisms of $A \oplus A^*$, with $\mathcal{S}_i$ given by (\ref{definition_Si}). The following assertions are equivalent:
\begin{enumerate}
    \item $(\omega_1, \omega_2, \omega_3)$ is a hypersymplectic structure with torsion on the Lie algebroid $(A,\mu)$, $\pi_i$ is the inverse of $\omega_i$ and $\pi_i$ is weak-Poisson, $i=1,2,3$;
    \item $(\mathcal{S}_1, \mathcal{S}_2, \mathcal{S}_3)$ is a hypersymplectic structure on the Courant algebroid $(A\oplus A^*, \mu+\psi)$, for some $\psi\in \Gamma(\wedge^3 A)$.
\end{enumerate}
\end{thm}
\begin{proof}
    In addition to the proof of Proposition~\ref{HS_preCourant_structure_mu+gamma+psi}, it is enough to notice that, since $(A, \mu)$ is a Lie algebroid,
    \begin{align*}
    \mu+\psi\text{ is Courant} &\Leftrightarrow \bb{\mu+\psi}{\mu+\psi}=0\\
        &\Leftrightarrow \bb{\mu}{\psi}=0 \Leftrightarrow \pi_i\text{ is weak-Poisson},\ i=1,2,3.
    \end{align*}

\end{proof}

In the next proposition we show that having a Lie bialgebroid $(A,A^*)$ equipped with a hypersymplectic structure with torsion on $A$ and a hypersymplectic structure with torsion on $A^*$ is equivalent to having a hypersymplectic structure (without torsion) on $A\oplus A^*$ equipped with a pre-Courant structure.

\begin{prop}\label{HS_preCourant_structure_mu+gamma+psi+phi}
Let $((A, A^*), \mu,\gamma)$ be a Lie bialgebroid, $(\omega_1, \omega_2, \omega_3)$ be a triplet of $2$-forms and $(\pi_1,\pi_2, \pi_3)$ be a triplet  of bivectors on $A$. Consider the triplet $(\mathcal{S}_1, \mathcal{S}_2, \mathcal{S}_3)$ of endomorphisms of $A \oplus A^*$, with $\mathcal{S}_i$ given by (\ref{definition_Si}). The following assertions are equivalent:
\begin{enumerate}
    \item $(\omega_1, \omega_2, \omega_3)$ is a hypersymplectic structure with torsion on the Lie algebroid $(A,\mu)$ and $(\pi_1, \pi_2, \pi_3)$ is a hypersymplectic structure with torsion on the Lie algebroid $(A^*,\gamma)$, with $\pi_i$ the inverse of $\omega_i$, $i=1,2,3$;
    \item $(\mathcal{S}_1, \mathcal{S}_2, \mathcal{S}_3)$ is a hypersymplectic structure on the pre-Courant algebroid $(A\oplus A^*, \mu+\gamma +\psi+\phi)$, for some $\psi\in \Gamma(\wedge^3 A)$ and $\phi \in \Gamma(\wedge^3 A^*)$.
\end{enumerate}
\end{prop}

\begin{proof}
Let us assume that $(\omega_1, \omega_2, \omega_3)$ is a hypersymplectic structure with torsion on $(A,\mu)$,
$(\pi_1, \pi_2, \pi_3)$ is a hypersymplectic structure with torsion on $(A^*,\gamma)$ and $\pi_i$ is the inverse of $\omega_i$, $i=1,2,3$.
From Lemma~\ref{lem_algebraic_conditions}, conditions i) and ii) of Definition~\ref{def_epsilonHS_pre-Courant} are satisfied while, for condition iii), we have:
\begin{align}
&\{\mathcal{S}_i, \{\mathcal{S}_i , \mu+\gamma +\psi + \phi \} \}= - \mu - \gamma -\psi -\phi \nonumber\\
 &\Leftrightarrow
 \begin{cases}
    \{\omega_i, \{ \omega_i, \psi \}\}- \{ \omega_i, \{\pi_i, \mu \}\}- \{\pi_i, \{\omega_i, \mu \}\}=-\mu\\
    - \{\omega_i, \{\{\pi_i, \gamma \}\}- \{\pi_i, \{\omega_i, \gamma \}\}+ \{\pi_i, \{\pi_i, \phi \}\}=-\gamma\\
    \{\omega_i, \{ \omega_i, \gamma \}\}- \{ \omega_i, \{\pi_i, \phi \}\}=-\phi\\
    \{\pi_i, \{\pi_i, \mu \}\}- \{\pi_i, \{\omega_i, \psi \}\}=-\psi
 \end{cases}\nonumber\\
 &\Leftrightarrow
 \begin{cases}
    \{\omega_i, \{ \omega_i, \psi \}\}=2  \{\pi_i, \{\omega_i, \mu \}\}\\
    \{\pi_i, \{\pi_i, \phi \}\}= 2 \{\omega_i, \{\pi_i, \gamma \}\}\\
    \{\omega_i, \{ \omega_i, \gamma \}\}=2 \phi\\
    \{\pi_i, \{\pi_i, \mu \}\}= 2 \psi
  \end{cases}\nonumber\\
 &\Leftrightarrow
 \begin{cases}
    \{\omega_i, \{ \omega_i, \gamma \}\}=2 \phi\\
    \{\pi_i, \{\pi_i, \mu \}\}= 2 \psi,
  \end{cases}\label{torsionSi_mu+gamma+psi+phi}
\end{align}
where the latter equivalence is given by Lemma~\ref{4equivalent2}. Equation (\ref{torsionSi_mu+gamma+psi+phi}) gives the appropriate definition of $\psi$ and $\phi$ in order to satisfy condition iii) of Definition~\ref{def_epsilonHS_pre-Courant}.

Now, we assume that $(\mathcal{S}_1, \mathcal{S}_2, \mathcal{S}_3)$ is a hypersymplectic structure on the pre-Courant algebroid $(A\oplus A^*,\mu + \gamma+\psi+ \phi)$. Using Lemma~\ref{lem_algebraic_conditions}, we conclude that $\pi_i$ is the inverse of $\omega_i$ and $N_i^2=-{\rm id}_A$, $i=1,2,3$. Moreover, from (\ref{torsionSi_mu+gamma+psi+phi}), we get $\psi= - \frac{1}{2}\{\pi_i, \{\mu, \pi_i\}\}$ and $\phi= - \frac{1}{2}\{\omega_i, \{\gamma, \omega_i\}\}$, for $i=1,2,3$. Thus,  $(\omega_1, \omega_2, \omega_3)$ is a hypersymplectic structure with torsion on the Lie algebroid $(A,\mu)$ and $(\pi_1, \pi_2, \pi_3)$ is a hypersymplectic structure with torsion on the Lie algebroid $(A^*,\gamma)$ (see (\ref{contravariant_def})).
\end{proof}

In the statement of Proposition \ref{HS_preCourant_structure_mu+gamma+psi+phi}, if we intend to obtain a Courant algebroid structure on $A\oplus A^*$, we need to impose some extra conditions, as shown in the next theorem.
%

\begin{thm} \label{last_thm}
Let $((A, A^*), \mu,\gamma)$ be a Lie bialgebroid, $(\omega_1, \omega_2, \omega_3)$ be a triplet of $2$-forms and $(\pi_1,\pi_2, \pi_3)$ be a triplet  of bivectors. Consider the triplet $(\mathcal{S}_1, \mathcal{S}_2, \mathcal{S}_3)$ of endomorphisms of $A \oplus A^*$, with $\mathcal{S}_i$ given by (\ref{definition_Si}). The following assertions are equivalent:
\begin{enumerate}
    \item $(\omega_1, \omega_2, \omega_3)$ is a hypersymplectic structure with torsion on the Lie algebroid $(A,\mu)$, $(\pi_1, \pi_2, \pi_3)$ is a hypersymplectic structure with torsion on the Lie algebroid $(A^*,\gamma)$, with $\pi_k$ the inverse of $\omega_k$, $k=1,2,3$, and $\{\gamma, \phi\}=\{\mu, \psi \}=\{\psi,\phi \}=0$, where $\psi= - \frac{1}{2}\{\pi_i, \{\mu, \pi_i\}\}$ and $\phi= - \frac{1}{2}\{\omega_j, \{\gamma, \omega_j\}\}$, for any $i,j\in\{1,2,3\}$;
    \item $(\mathcal{S}_1, \mathcal{S}_2, \mathcal{S}_3)$ is a hypersymplectic structure on the Courant algebroid $(A\oplus A^*, \mu+\gamma +\psi+\phi)$, for some $\psi\in \Gamma(\wedge^3 A)$ and $\phi \in \Gamma(\wedge^3 A^*)$.
\end{enumerate}
\end{thm}
\begin{proof}
    First notice that, using the fact that $((A, A^*), \mu,\gamma)$ is a Lie bialgebroid, we have the following equivalences:
    \begin{align}
    \mu+\gamma +\psi+\phi\text{ is Courant}
    &\Leftrightarrow \{\mu + \gamma+\psi+ \phi, \, \mu + \gamma+\psi+\phi \}=0\nonumber\\
    &\Leftrightarrow \,
    \begin{cases}
        \{\mu, \mu \}=-2 \{\gamma, \phi\}\\
        \{\gamma, \gamma \}=-2\{\mu, \psi \}\\
        \{\mu, \gamma \}= - \{\psi,\phi \}\\
        \{\gamma, \psi \}=0\\
        \{ \mu, \phi \}=0
    \end{cases}\nonumber\\
    &\Leftrightarrow \, \{\gamma, \phi\}=\{\mu, \psi \}=\{\psi,\phi \}=\{\gamma, \psi \}=\{ \mu, \phi \}=0.\label{aux}
    \end{align}

    Let us assume assertion ii), then $\mu+\gamma +\psi+\phi$ is a Courant structure and condition (\ref{aux}) is satisfied. In particular, $\{\gamma, \phi\}=\{\mu, \psi \}=\{\psi,\phi \}=0$. Furthermore, the same computations we have done in the proof of Proposition \ref{HS_preCourant_structure_mu+gamma+psi+phi} (see (\ref{torsionSi_mu+gamma+psi+phi})), yield $\psi=-\frac{1}{2} \{ \pi_i, \{ \mu, \pi_i \} \}$ and $\phi=-\frac{1}{2} \{ \omega_i, \{ \gamma, \omega_i \} \}$. The remaining part of assertion i) is a consequence of Proposition \ref{HS_preCourant_structure_mu+gamma+psi+phi}.

    Assuming now assertion i), let us prove ii). Taking into account Proposition \ref{HS_preCourant_structure_mu+gamma+psi+phi}, we only need to prove that $\mu+\gamma +\psi+\phi$ is a Courant structure, i.e., that (\ref{aux}) is satified. Because part of (\ref{aux}) holds by assumption, we only need to prove $\{\gamma, \psi \}=\{ \mu, \phi \}=0$. We shall compute one equality, the other is similar. We have
    \begin{eqnarray*}
        \{\gamma, \psi \}&=& -\frac{1}{2} \{\gamma, \{ \pi_i, \{ \mu, \pi_i \} \} \}\\
        &=& -\frac{1}{2} \{ \{ \gamma, \pi_i \}, \{ \mu, \pi_i \} \}-\frac{1}{2} \{\pi_i, \{\gamma, \{ \mu, \pi_i \}\}\}\\
        &=&   -\frac{1}{2} \{ \{ \gamma, \pi_i \}, \{ \mu, \pi_i \} \}-\frac{1}{2} \{\pi_i, \{\mu, \{ \gamma, \pi_i \}\}\}=0,
    \end{eqnarray*}
    where we used the Jacobi identity of $\{.,.\}$ and the fact that $((A, A^*), \mu, \gamma)$ is a Lie bialgebroid (in particular that $\{\mu, \gamma \}=0$).
\end{proof}

If we take  $\phi=0$ in Theorem~\ref{last_thm}, then the Lie algebroid $(A^*, \gamma)$ is equipped with a hypersymplectic structure (without torsion) determined by $(\pi_1, \pi_2, \pi_3)$.
So, Theorem~\ref{last_thm} shows that having a Lie bialgebroid $(A,A^*)$ equipped with a hypersymplectic structure with torsion on $A$ and a hypersymplectic structure on $A^*$ is equivalent to having a hypersymplectic structure on the Courant algebroid $(A\oplus A^*, \mu+\gamma+\psi)$, which is the double of the quasi-Lie bialgebroid $((A,A^*), \mu, \gamma, \psi)$.

\

\noindent {\bf Acknowledgments.} This work was partially supported by CMUC-FCT (Portugal) and FCT grant PEst-C/MAT/UI0324/2013 through European program COMPETE/FEDER.

\end{document}